\documentclass[10pt]{article}

\usepackage{etex}
\usepackage[margin={1.2in}]{geometry}
\usepackage[titletoc]{appendix}

\usepackage{tikz}
\usetikzlibrary{matrix}
\usepackage{eucal}
\usepackage{mathrsfs}
\usepackage{stmaryrd}

\usepackage{rotating}

\usepackage{longtable}

\usepackage[sc]{mathpazo}
\usepackage{courier}

\usepackage[utf8]{inputenc}
\usepackage[T1]{fontenc}

\usepackage[english]{babel}
\usepackage{blindtext}

\usepackage{amsmath, amssymb}

\usepackage{microtype}

\usepackage{amsmath}
\usepackage{amssymb}
\usepackage{amsthm}
\usepackage{color}
\usepackage{latexsym,extarrows}

\usepackage[all]{xy}

\usepackage[stable,multiple]{footmisc}

\RequirePackage[font=small,format=plain,labelfont=bf,textfont=it]{caption}
\makeatletter
\newsavebox{\@brx}
\newcommand{\llangle}[1][]{\savebox{\@brx}{\(\m@th{#1\langle}\)}%
  \mathopen{\copy\@brx\kern-0.5\wd\@brx\usebox{\@brx}}}
\newcommand{\rrangle}[1][]{\savebox{\@brx}{\(\m@th{#1\rangle}\)}%
  \mathclose{\copy\@brx\kern-0.5\wd\@brx\usebox{\@brx}}}
\makeatother

\begin{document}
\def\e#1\e{\begin{equation}#1\end{equation}}
\def\ea#1\ea{\begin{align}#1\end{align}}
\def\eq#1{{\rm(\ref{#1})}}
\theoremstyle{plain}
\newtheorem{thm}{Theorem}[section]
\newtheorem{lem}[thm]{Lemma}
\newtheorem{prop}[thm]{Proposition}
\newtheorem{cor}[thm]{Corollary}
\theoremstyle{definition}
\newtheorem{dfn}[thm]{Definition}
\newtheorem{ex}[thm]{Example}
\newtheorem{rem}[thm]{Remark}
\newtheorem{conjecture}[thm]{Conjecture}

\newcommand{\D}{\mathrm{d}}
\newcommand{\A}{\mathcal{A}}
\newcommand{\LL}{\llangle[\Big]}
\newcommand{\RR}{\rrangle[\Big]}
\newcommand{\LD}{\Big\langle}
\newcommand{\RD}{\Big\rangle}
\newcommand{\F}{\mathcal{F}}
\newcommand{\HH}{\mathcal{H}}
\newcommand{\X}{\mathcal{X}}
\newcommand{\PP}{\mathbb{P}}
\newcommand{\K}{\mathscr{K}}
\newcommand{\q}{\mathbf{q}}

\newcommand{\op}{\operatorname}
\newcommand{\C}{\mathbb{C}}
\newcommand{\N}{\mathbb{N}}
\newcommand{\R}{\mathbb{R}}
\newcommand{\Q}{\mathbb{Q}}
\newcommand{\Z}{\mathbb{Z}}
\renewcommand{\H}{\mathbf{H}}

\newcommand{\Etau}{\text{E}_\tau}
\newcommand{\E}{{\mathcal E}}
\newcommand{\G}{\mathbf{G}}
\newcommand{\eps}{\epsilon}
\newcommand{\g}{\mathbf{g}}
\newcommand{\im}{\op{im}}

\newcommand{\h}{\mathbf{h}}

\newcommand{\Gmax}[1]{G_{#1}}
\newcommand{\AW}{E}
\providecommand{\abs}[1]{\left\lvert#1\right\rvert}
\providecommand{\norm}[1]{\left\lVert#1\right\rVert}
\newcommand{\abracket}[1]{\left\langle#1\right\rangle}
\newcommand{\bbracket}[1]{\left[#1\right]}
\newcommand{\fbracket}[1]{\left\{#1\right\}}
\newcommand{\bracket}[1]{\left(#1\right)}
\newcommand{\ket}[1]{|#1\rangle}
\newcommand{\bra}[1]{\langle#1|}

\newcommand{\ora}[1]{\overrightarrow#1}

\providecommand{\from}{\leftarrow}
\newcommand{\bl}{\textbf}
\newcommand{\mbf}{\mathbf}
\newcommand{\mbb}{\mathbb}
\newcommand{\mf}{\mathfrak}
\newcommand{\mc}{\mathcal}
\newcommand{\cinfty}{C^{\infty}}
\newcommand{\pa}{\partial}
\newcommand{\prm}{\prime}
\newcommand{\dbar}{\bar\pa}
\newcommand{\OO}{{\mathcal O}}
\newcommand{\hotimes}{\hat\otimes}
\newcommand{\BV}{Batalin-Vilkovisky }
\newcommand{\CE}{Chevalley-Eilenberg }
\newcommand{\suml}{\sum\limits}
\newcommand{\prodl}{\prod\limits}
\newcommand{\into}{\hookrightarrow}
\newcommand{\Ol}{\mathcal O_{loc}}
\newcommand{\mD}{{\mathcal D}}
\newcommand{\iso}{\cong}
\newcommand{\dpa}[1]{{\pa\over \pa #1}}
\newcommand{\Kahler}{K\"{a}hler }
\newcommand{\0}{\mathbf{0}}

\newcommand{\B}{\mathcal{B}}
\newcommand{\V}{\mathcal{V}}

\newcommand{\M}{\mathfrak{M}}

\renewcommand{\Im}{\op{Im}}
\renewcommand{\Re}{\op{Re}}

\newcommand{\DD}{\Omega^{\text{\Romannum{2}}}}


\numberwithin{equation}{section}

\newcommand{\comment}[1]{\textcolor{red}{[#1]}} 

\makeatletter
\newcommand{\subjclass}[2][2020]{%
  \let\@oldtitle\@title%
  \gdef\@title{\@oldtitle\footnotetext{#1 \emph{Mathematics Subject Classification.} #2}}%
}
\newcommand{\keywords}[1]{%
  \let\@@oldtitle\@title%
  \gdef\@title{\@@oldtitle\footnotetext{\emph{Key words and phrases.} #1.}}%
}
\makeatother

\makeatletter
\let\orig@afterheading\@afterheading
\def\@afterheading{%
   \@afterindenttrue
  \orig@afterheading}
\makeatother

\title{\bf Twisted Sectors in Calabi-Yau Type Fermat Polynomial Singularities  and Automorphic Forms}
\author{Dingxin Zhang and Jie Zhou}
\date{}
\maketitle

\begin{abstract}
We study one-parameter deformations  of
Calabi-Yau type Fermat polynomial singularities along degree-one directions.
We show that twisted sectors in the vanishing cohomology are components of automorphic forms for certain triangular groups.
We prove consequently that genus zero Gromov-Witten generating series
of the corresponding Fermat Calabi-Yau varieties are components of automorphic forms.
The main tools we use are
mixed Hodge structures for quasi-homogeneous polynomial singularities, Riemann-Hilbert correspondence,
and genus zero mirror symmetry.
\end{abstract}

\setcounter{tocdepth}{2} \tableofcontents

\section{Introduction}

The interplay between Gromov-Witten (GW) theory (and other enumerative theories) and modular forms
(and more generally automorphic forms)
has been an interesting topic in the last few decades.
On the one hand, the identification between generating series of enumerative invariants
and automorphic forms
can often reduce computations on the infinite sequences
of invariants to finite computations and hence make it much easier for computational purposes.
On the other hand, the nice properties of automorphic forms
provide useful tools and perspectives to understand these invariants,
and to test new ideas and conjectures in enumerative theories.

Our main idea in approaching modularity and automorphicity in the generating series of GW invariants (and other enumerative invariants)
is illustrated in Figure  \ref{figureretraction} below.
\begin{figure}[h]
  \renewcommand{\arraystretch}{1}
\begin{displaymath}
\xymatrix{
&\text{A-model}\, \ar@{->}[ddr]_{\text{Gromov-Witten axioms}} \ar@{<-->}[rr]^{\text{mirror symmetry} }
 &&
 \ar@{->}[ldd]^{\textrm{mixed Hodge structures} }
 \text{B-model}
\\
&&&&\\
&&\text{D-modules: differential equations} \ar@{->}[dd]^{\text{Riemann-Hilbert correspondence}} &&\\
&&&&\\
&&\text{representations: automorphic bundles}&&\\
 }
\end{displaymath}
 \caption[automorphicbundles]{Automorphic bundles arise in both A-model and B-model.
Identification of the corresponding D-modules arising in the two models provides a mirror map. }
  \label{figureretraction}
\end{figure}

In the present work, for the A-model we shall study
GW invariants of the Calabi-Yau (CY) variety
\begin{equation}\label{eqnCYA}
M=\{p(x):=\sum_{i=0}^{n}x_{i}^{n+1}=0\}\subseteq \mathbb{P}^{n}\,,\quad n\geq 2\,,
\end{equation}
 and certain quotients thereof.
 The invariants and corresponding
genus zero mirror symmetry for the CY variety $M$
have been
extensively explored
\cite{Candelas:1990rm, Givental:1996, Givental:1998, Lian:1997, Lian:1998},
our main focus is then the orbifold GW invariants of the quotients of $M$, see e.g., \cite{Lee:2014, Iritani:2021} for related studies.

For the B-model, we consider the mixed Hodge structure on the vanishing cohomology of the polynomial singularity
\begin{eqnarray}\label{eqnDworksingularity}
f_a: \mathbb{C}^{n+1}&\rightarrow& \mathbb{C}\,,\nonumber\\
(z_0,z_1,\cdots, z_{n})&\mapsto& \sum_{i=0}^{n}z_{i}^{n+1}-a\prod_{i=0}^{n}z_{i}\,,\quad a\in S\,,
\end{eqnarray}
where $S:=\mathbb{C}-\{a^{n+1}=(n+1)^{n+1}\}$ is contained in the $\mu$-constant strata.
Closely related is the CY variety
\begin{equation}\label{eqnCYB}
Q_{a}=\{f_a=0\} \subseteq \mathbb{P}^{n}\,,\quad a\in S\,.
\end{equation}
The corresponding family of singularities or  CY varieties will be referred to as the \emph{Dwork family} in this work.
The vanishing cohomology of the singularity $f_{a}$ contains
a subspace $D_{\bold{0}}$
that is identical to the primitive part of  $H^{n-1}(Q_{a})$
under the Poincar\'e residue map (see Section \ref{sectwistedsectorsmhs}).
This space is mirror to a subspace of $H^{\mathrm{even}}(M)$
and connects to the genus zero GW theory of the CY space $M$.

The polynomial \eqref{eqnDworksingularity} is quasi-homogeneous with
weighted degree $q_{i}$ (called polynomial degree in what follows) of $z_{i}$ given by $q_{i}=1/(n+1)$ for $i=0,1,\cdots, n$.
It is of Calabi-Yau type in the sense that
the summation of the individual weighted degrees $q_{i}$ is $1$:
$\sum_{i=0}^{n}q_{i}=1$.
The quasi-homogeneity leads to great simplifications in the studies of mixed Hodge structures for the singularities as we shall see.

\subsubsection*{Automorphic bundles from vanishing cohomology}

``Twisted sectors'' attached to a singularity are defined by using the mixed Hodge structure on the vanishing cohomology.
They are represented by (leading parts of) geometric sections \cite{Arnold:1985, Arnold:1988}.
For a family of quasi-homogeneous polynomial singularities, they have very concrete descriptions. See Section \ref{sectwistedsectorsmhs}
for their definitions and descriptions.

Consider the Dwork family of singularities
\eqref{eqnDworksingularity}.
A direct calculation shows that each geometric section $s[\bold{z}^{\bold{m}}\Omega]$, labelled by a vector $\bold{m}$ (see Definition \ref{dfntwistedsectormonomials}), satisfies
a simple differential equation
and thus gives rise to a $D$-module $D_{\bold{m}}$.
This also gives a flat vector bundle, over a base $\mathcal{M}$ that is related to $S$ by a finite covering map.
By
the Riemann-Hilbert correspondence (more concretely the monodromy representation),
one obtains
an orbifold structure\footnote{To be precise,
the map to the coarse moduli is $T_{\bold{m}}\rightarrow \overline{\mathcal{M}}_{\bold{m}}$, where
$\overline{\mathcal{M}}_{\bold{m}}$ is obtained from $\mathcal{M}$
by filling in the elliptic points (i.e., points with finite-order monodromies). See Section \ref{secRHcorrespondence}
for detailed explanation.
Here we ignore this subtlety for notational convenience.}
$T_{\bold{m}}=\Gamma_{\bold{m}}\backslash\mathcal{H}$ on $\mathcal{M}$, where
$\Gamma_{\bold{m}}<\mathrm{PSL}_{2}(\mathbb{R})$ is a triangular group that acts on the upper-half plane $\mathcal{H}$.
The multi-valued map $s_{\bold{m}}: \mathcal{M}\rightarrow \mathcal{H}$ is provided by the Schwarzian (i.e., the ratio of the logarithmic solution and the regular solution near $a=\infty$ of the corresponding differential equation)
and will be called
the Schwarzian uniformization in this work.
Each cohomology sector  $D_{\bold{m}}$ then
corresponds to a monodromy representation
\begin{equation}\label{eqnintroautomorphicrepresentation}
\rho_{\bold{m}}: \,\Gamma_{\bold{m}}\cong \pi_{1}^{\mathrm{orb}}(T_{\bold{m}})\rightarrow \mathrm{End}\, (D_{\bold{m}})\,.
\end{equation}
The bundle $D_{\bold{m}}$ is an automorphic bundle for the triangular group $\Gamma_{\bold{m}}$
in the sense of the definition below.
\begin{dfn}\label{defautomorphicbundle}
Let $\Gamma<\mathrm{PSL}_{2}(\mathbb{R})=\mathrm{Aut}\,\mathcal{H}$ be a Fuchsian group acting on the upper-half plane $\mathcal{H}$.
Let $e :\Gamma\times \mathcal{H}\rightarrow \mathrm{GL}(\mathbb{V})$ be an automorphy factor, that is,
\begin{equation*}
e_{\gamma_{2}}(\gamma_{1}\tau)\circ e_{\gamma_1}(\tau)=e_{\gamma_{2}\circ \gamma_{1}}(\tau)\,,
\quad
\forall\,
\gamma_{1},\gamma_{2}\in \Gamma\,, \tau\in \mathcal{H}\,.
\end{equation*}
The bundle $\mathcal{V}_{e}:=\mathcal{H}\times_{e}\mathbb{V}$ over $\Gamma\backslash\mathcal{H}$
defined by $e\in H^{1}(\Gamma, \mathrm{GL}( \mathcal{O}(\mathcal{H}) \otimes \mathbb{V}))$
is called an automorphic bundle for $\Gamma$.
A (vector-valued) automorphic form with automorphy factor
$e$ is a meromorphic section of the automorphic bundle $\mathcal{V}_{e}$.
Concretely, it is given by a meromorphic function
$\phi:\mathcal{H}\rightarrow \mathbb{V}$ satisfying
\begin{equation}\label{eqnanalyticautomorphicform}
\phi(\gamma\tau)=e_{\gamma}(\tau)\phi(\tau)\,,
\quad
\forall\,
\gamma\in \Gamma\,, \tau\in \mathcal{H}\,.
\end{equation}
Similarly, one can define automorphic bundles
over the compactification of $\Gamma\backslash\mathcal{H}$
by specifying meromorphic behaviour at the cusps on $\Gamma\backslash\mathcal{H}$.
\end{dfn}
\begin{ex}
Taking the automorphy factor $e$ to be a representation $\rho:\Gamma\rightarrow \mathrm{GL}(\mathbb{V})$, then the automorphic bundle $\mathcal{V}_{e}$ is a flat bundle.
Taking $e$ to be the $j$-automorphy factor
\begin{equation}\label{eqnjautomorphy}
j(\gamma, \tau)=(c\tau+d)^{k}\,,\quad
\forall\, \gamma=
\begin{pmatrix}
a & b\\
c & d
\end{pmatrix}\in \Gamma\,,~k\in 2 \mathbb{Z}\,,
\end{equation}
then one obtains the notion of automorphic forms of weight $k$.
For a general $k\in \mathbb{Q}\,$,
 $(c\tau+d)^{k}$ is multi-valued, and needs to be accompanied by a multiplier system to yield an automorphy factor  \cite{Sch:1974, Rankin:1977ab}.
\end{ex}

One can furthermore find a triangular
group $\Gamma_{\star}=\Gamma_{n+1,\infty,\infty}<\mathrm{PSL}_{2}(\mathbb{R})$
such that for any $\bold{m}$ there exists a natural homomorphism
\begin{equation}\label{eqnauxiliarygroup}
\rho_{ \bold{m},\star}: \Gamma_{\star} \rightarrow \Gamma_{\bold{m}}\,.
\end{equation}
which respects the semi-simple parts of their  monodromy representations. See Section \ref{secexamples} for detailed discussions.
Then the composition
$\rho_{\bold{m}} \circ \rho_{\bold{m},\star}$ realizes $D_{\bold{m}}$ as a representation of $\Gamma_{\star} $. This tells that
all cohomology bundles are in fact automorphic bundles for $\Gamma_{\star} $.

\begin{thm}[Theorem \ref{thmautomorphicform}]\label{thmautomorphicformintro}
Consider the Dwork family of polynomial singularities \eqref{eqnDworksingularity}.
For each $\bold{m}$, there exists a triple $\bold{\ell}(\bold{m})=(\ell_{0}(\bold{m})=n+1, \ell_{1}(\bold{m}),\ell_{\infty}(\bold{m}))$,
such that
the corresponding flat bundle $D_{\bold{m}}$ over $\mathcal{M}$
is an automorphic bundle
 and thus
the geometric section
$s[\bold{z}^{\bold{m}}\Omega]$ is the component of an automorphic form
for the triangular group $\Gamma_{\ell_0(\bold{m}),\ell_1(\bold{m}),\ell_\infty(\bold{m})}<\mathrm{PSL}_{2}(\mathbb{R})$.
In particular, all of the geometric sections are components of automorphic forms for $\Gamma_{n+1,\infty,\infty}$.

\end{thm}

By Deligne's canonical extension, the flat bundle $D_{\bold{m}}$ extends to a vector bundle $\overline{D}_{\bold{m}}$ on
the compactified orbifold $\overline{T}_{\bold{m}}$.
Both automorphic bundles $D_{\bold{m}}, \overline{D}_{\bold{m}}$ are equipped with  \cite{Schmid:1973}  Hodge filtrations $F^{\bullet}D_{\bold{m}}, F^{\bullet}\overline{D}_{\bold{m}}$ which are induced from the ones
defined on the vanishing cohomology bundle $H$.
By definition, sections of $D_{\bold{m}}, \overline{D}_{\bold{m}} $
 and in particular those of the sub-bundles $F^{p}D_{\bold{m}}, F^{p} \overline{D}_{\bold{m}}$,
  are components of automorphic forms with or without boundary conditions imposed at the elliptic points and cusps
\cite{Saber2014:vectorvalued, Candelori:2017vector, Candelori:2019vector}.

Moreover,
for the cases with $\mathrm{rank}\,D_{\bold{m}}\leq 2$,
the automorphy factors for the nontrivial top pieces $F^{\mathrm{top}}D_{\bold{m}}$
turn out to be given by the familiar $j$-automorphy factor \eqref{eqnjautomorphy}
by using
the Schwarzian uniformization, and the corresponding automorphic forms have weights given by $\mathrm{rank}\,D_{\bold{m}}-1$.
A frame of any other Hodge sub-bundle can then be obtained by applying the Gauss-Manin connection on that of $F^{\mathrm{top}}D_{\bold{m}}$.
From the resulting frame one can directly read off the
automorphy factor for the sub-bundle.
Sections of these Hodge sub-bundles are analogues of quasi-modular forms \cite{Kaneko:1995}. See Remark \ref{remquasimodularforms} and Section \ref{secexamples}
for detailed discussions on these.

\subsubsection*{Automorphicity of genus zero Gromov-Witten generating series}

The sector $D_{\bold{m}}$ with $\bold{m}=\bold{0}$ plays a distinguished role
for the following reasons.
\begin{itemize}
\item Consider the CY variety \eqref{eqnCYB}. As mentioned earlier,
the $D_{\bold{0}}$ sector in the vanishing cohomology of $f_a$ is identified with the primitive part $H^{n-1}(Q_{a})$
via the Poincar\'e residue map. 
This establishes the so-called Landau-Ginzburg/Calabi-Yau correspondence \cite{Chiodo:2010}
on the level of state spaces.
Under mirror symmetry \cite{Chiodo:2010}, it leads the identification of  the genus zero Gromov-Witten theory of $M$ in \eqref{eqnCYA} and the genus zero
Fan-Jarvis-Ruan-Witten (FJRW) theory of $p:\mathbb{C}^{n+1}\rightarrow \mathbb{C}$.

\item
Period integrals involving other sectors
$D_{\bold{m}}$ are often expanded
in terms of the Schwarzian $s_{\bold{0}}\in \mathcal{H}_{\bold{0}}$ arising from the sector $D_{\bold{0}}$, instead of  their own Schwarzians
$s_{\bold{m}}$ valued in a copy of upper-half plane $\mathcal{H}_{\bold{m}}$
labelled by $\bold{m}$.
From the perspective of the GW theory of the  mirror $M$, this is due to the divisor axiom.
Note that here although both $\mathcal{H}_{\bold{0}}$ and $\mathcal{H} _{\bold{m}}$
are isomorphic to the upper-half plane $\mathcal{H}$, the corresponding Schwarzian maps on $\mathcal{M}$ are not identical.
\end{itemize}
In the cases of Fermat cubic and quartic, we can in fact show (see Section \ref{secexamples}
 and  Appendix \ref{secmarginalquartic}) that the Schwarzian $s_{\bold{0}}$ corresponds to
the complex structure modulus $\tau$ for certain elliptic curve families.

Let us be a little more precise.
Consider the genus zero GW theory of the CY variety $M$.
According to genus zero mirror symmetry \cite{Givental:1996, Lian:1997} for CY hypersurfaces in projective spaces,
it is mirror to the
theory of variation of Hodge structures
for the degeneration near $a=\infty$ in the Dwork family $\{Q_{a}\}_{a\in S}$ in \eqref{eqnCYB}\footnote{The family of mirror manifolds of $M$ is in fact the quotient of the Dwork family \cite{Batyrev:1994hm},
whose associated variation of Hodge structures is however equivalent to that of the Dwork family itself. See \cite{Cox:1999} and \cite{Gross:2003} for nice expositions.}, and the generating series of 
genus zero one-point GW invariants of $M$, encoded by the Givental $I$-function,
is mirror to the period integral of the CY form in the Dwork family.
By Theorem \ref{thmautomorphicformintro}
and the Poincar\'e residue map mentioned above, we obtain the following result.
\begin{thm}[Theorem \ref{thmgenuszeroGWautomorphicform}]
Consider the genus zero Gromov-Witten invariants for the Calabi-Yau $(n-1)$-fold $M$
defined by the vanishing locus of the Fermat polynomial $\sum_{i=0}^{n}x_{i}^{n+1}=0$ in $\mathbb{P}^{n}$.
Then
the $I$-function
is a component of an automorphic form valued in $H^{*}(M)$ for the triangular group
$\Gamma_{n+1,\infty,\infty}$.
In particular, for the $n=2,3$ cases, the automorphic form is a vector-valued elliptic modular form (with possibly nontrivial multiplier systems) for certain congruence subgroup in
$\mathrm{PSL}_{2}(\mathbb{Z})$.
\end{thm}

The $D_{\bold{0}}$ sector  is  equipped with a variation of Hodge structures (VHSs) by the theory of mixed Hodge structure on the vanishing cohomology.
The Poincar\'e residue map mentioned above actually respects the VHSs, see Section \ref{secmixedHodgestructuresonvanishingcohomology}.
This sector
captures not only genus zero data but also essential higher genus information \cite{Bershadsky:1993cx, Witten:1993ed, Li2011:calabi, Costello:2012cy, Yamaguchi:2004bt}.
We demonstrate this on the Fermat quintic case which has been regarded as the prototypical example of mirror symmetry
since its birth \cite{Greene:1990Duality, Candelas:1990Calabi, Candelas:1990rm}.
Consider the Picard-Fuchs operator for the Dwork family $\{Q_{a}\}_{a\in S}$ in \eqref{eqnCYB}
\begin{equation}
\mathcal{D}_{\mathrm{quintic}}=\theta^{4}-z\prod_{k=1}^{4} (\theta+{k\over 5})\,,
\quad  z=5^{-5}a^{-5}\,,\quad \theta:={z}{\partial\over \partial z}\,.
\end{equation}
Take a basis of solutions near the maximally unipotent monodromy point $z=0$ and denote them by $I_{0},I_{1},I_{2},I_{3}$, respectively.
Set
$T={I_{1}/ I_{0}}$ and define
the Yamaguchi-Yau polynomial ring
\cite{Bershadsky:1993cx, Yamaguchi:2004bt}  to be (see Definition \ref{dfnYamaguchiYauring})
\begin{equation}
\mathcal{F}_{\mathrm{YY}}:=\mathbb{C}(z)[
{I_{0}'\over I_{0}},
{I_{0}''\over I_{0}},
{I_{0}'''\over I_{0}},
{T''\over T'}]\,,\quad '=\partial_{z}\,.
\end{equation}
A renowned conjecture \cite{Bershadsky:1993cx, Yamaguchi:2004bt}  in the studies of higher genus GW theory
that has stimulated many progresses in the last few decades is
that GW generating series of higher genus invariants
of the Fermat quintic are polynomials lying in the Yamaguchi-Yau ring $\mathcal{F}_{\mathrm{YY}}$.
This conjecture is recently proved in a series of works including
 \cite{Zinger:2009, Lho:2018holomorphic, Chang:2018polynomial, Guo:2018structure},
 in which properties of the Yamaguchi-Yau polynomial ring plays a central role.

 In this work we show the following result regarding the automorphicity of this ring
 which seems to be of independent interest.
\begin{thm}[Proposition \ref{propYYdifferentialring}, Theorem \ref{thmalgebraicindependence}]\label{thmYY}
Consider the Dwork family \eqref{eqnCYB} with $n=4$.
The Yamaguchi-Yau ring $\mathcal{F}_{\mathrm{YY}}$ is a differential ring
under $\partial_z$, with generators being components of automorphic forms for $\Gamma_{\infty,\infty,5}$.
Furthermore, the Yamaguchi-Yau generators are algebraically independent over $\mathbb{C}(z)$.
\end{thm}

The above result on period integrals can be rephrased as follows.
Since $S$ is affine, we may freely pass between sheaves and ordinary modules by taking global sections and localizing.
Let $(H^{n-1}_{\mathrm{prim}}, \mathcal{V}, F^{\bullet}, \nabla, S)$ be the aforementioned variation of Hodge structures associated to the degeneration in the
Dwork family \eqref{eqnCYB}.
Consider the fractional field of the coordinate expressions of  elements in $H^{0}(S, \mathcal{V})$ with respect to the locally constant frame,
then the
proof of Theorem \ref{thmYY}, in particular \eqref{eqnrelationsforYY}, gives the set of generators and relations for its differential closure.

\begin{rem}
According to \eqref{eqnanalyticautomorphicform} in Definition \ref{defautomorphicbundle},
analytically automorphic forms are vector-valued functions on $\mathcal{H}$ that enjoy nice symmetries under the action of $\Gamma$ specified by the automorphy factors.
Once the generating series in enumerative theories
 are proved to be automorphic forms, one can then use
these symmetries to study relations between various enumerative theories defined over $\mathcal{M}$ such as
higher genus GW and FJRW theories  \cite{Bas:2018, Shen:2018, Li2020:higher}. These relations are
usually difficult to establish by other methods such as analytical continuation.\\
\end{rem}

Many discussions in this work in fact apply to more general one-parameter deformations
and to  more general CY type quasi-homogeneous polynomial singularities,
  than the one given in \eqref{eqnDworksingularity}, using the results  from \cite{Beukers:1989, Steebrink:1976}.
In the latter situation the discussions are certainly more involved
  due to the possible singularities at infinity.
We however have decided to focus only on Fermat polynomials and the
one-parameter deformations \eqref{eqnDworksingularity},
in order to keep the presentation relatively clean.
We also remark that for multi-dimensional deformations we do not have a good understanding on
the automorphic forms part yet, and wish to study them in a future investigation.

Quite a few results contained in this work are known to experts, especially those on the differential equations governing various sectors that have kept occurring in both
GW theory and singularity theory. The main contribution of this work is,
through  carefully working out the details for several examples,
 to offer a Hodge-theoretic formulation
for some of the problems in the studies of GW theory and its mirror symmetry, and to strengthen the connection between
GW generating series and
automorphic forms which recently seems to have provided useful perspectives to enumerative geometry.

\subsection*{Organization of the paper}

In Section \ref{secDmoduleandautomorphicbundles} we study the D-modules underlying various  cohomology sectors in the Dwork family, and
the corresponding automorphic bundles determined from the Riemann-Hilbert correspondence.
We supply some examples in which the computations are carefully worked out.
In addition, for the Fermat quintic we use differential Galois theory to study algebraic properties of the  Yamaguchi-Yau ring.

In Section \ref{secmirrorsymmetry} we consider genus zero mirror symmetry of Fermat Calabi-Yau varieties.
Applying the results in Section \ref{secDmoduleandautomorphicbundles}  we show that genus zero GW generating series are components of automorphic forms.
We focus on quotients of the Fermat quartic K3 surface as examples, by showing that twisted sectors in the Chen-Ruan
cohomology are matched with twisted sectors of the mirror singularity theory and the corresponding GW generating series are matched with oscillating integrals.

Some preliminaries on singularity theory are collected in Appendix  \ref{secpreliminariesonsingularities}.
Detailed discussions on elliptic modular forms
appearing in the Fermat quartic case are carried out in Appendix \ref{secmarginalquartic}.

\subsection*{Acknowledgement}

We thank Todor Milanov, Yongbin Ruan, Yefeng Shen, Hsian-Hua Tseng and Yingchun Zhang for useful conversations and communications.
This work is supported by national key research and development
program of China No. 2022YFA1007100 and No. 2020YFA0713000 as well as NSFC No.~12371011.
The work of J.~Z. is also partially supported by
a start-up grant at Tsinghua University and the Young overseas high-level talents introduction plan of China.

\section{Cohomology bundles and automorphic forms}
\label{secDmoduleandautomorphicbundles}

In this section,
we first explain the definition of twisted sectors using the mixed Hodge structure on the vanishing cohomology of a quasi-homogeneous singularity.
Then we consider special one-parameter families of Calabi-Yau type Fermat polynomial singularities, with an emphasis on the Dwork family
\eqref{eqnDworksingularity}.
We compute the D-modules corresponding to the twisted sectors of the latter
and
conclude that they correspond to automorphic bundles for certain
triangular groups.
Our main tools are the theory of mixed Hodge structures for quasi-homogeneous polynomial singularities and
Riemann-Hilbert correspondence.

We follow closely the notation and expositions in \cite {Kulikov:1998}.
For completeness, some basics on singularities are collected in
Appendix \ref{secpreliminariesonsingularities}, following
\cite{Saito:1983, Matsuo:1998, Kulikov:1998, Saito:08}.
Readers who are unfamiliar with singularity theory are referred there for a quick introduction.

\subsection{Twisted sectors for quasi-homogeneous singularities}
\label{sectwistedsectorsmhs}

Let $z=(z_0,z_1,\cdots ,z_{n})$
be the standard coordinates on $\mathbb{C}^{n+1}$ and
$f:\mathbb{C}^{n+1}\rightarrow \mathbb{C}$ be a quasi-homogeneous polynomial,
which has an isolated singularity at the origin,
with the polynomial degree of $z_{i}$ given by $q_{i},i=0,1,\cdots,n$.
Consider the one-parameter deformation
\begin{equation}\label{eqnFdfn}
F:\, X=\mathbb{C}^{n+1}\times S\rightarrow \mathbb{C}\,,
\quad
(z, a)\mapsto f(z)-a e_1(z)\,,
\end{equation}
where
\begin{itemize}
\item
the monomial
\begin{equation}\label{eqnoneparametermarginaldeformation}
e_{1}(z)=\prod_{i=0}^{n} z_{i}^{d_{i}}\,,\quad d_{i}\in \mathbb{N}
\end{equation}
is of polynomial degree $1$. For example, for the Fermat polynomial $f= \sum_{i=0}^{n}z_{i}^{n+1}$, one can take $e_1(z)=\prod_{i=0}^{n}z_{i}$
which then gives the Dwork family \eqref{eqnDworksingularity}.
\item the deformation space is
\begin{equation}
S=\mathbb{C}-\Delta\,,
\end{equation}
with $\Delta$ the discriminant locus, such that the projective variety $Q_{a},a\in S$
defined as the vanishing locus of $f_{a}(z):=f(z)-ae_{1}(z)$ in the corresponding weighted projective space is nonsingular.
\end{itemize}
 Then one can check that $S$
lies in the $\mu$-constant stratum.\footnote{
This ensures that in the setting of Appendix \ref{secpreliminariesonsingularities}
one has that
$X= \mathbb{C}^{n+1}\times S$
 with $q:X\rightarrow S$ and $\varphi=(F,q):X\rightarrow \mathbb{C}\times S$, with no need to restrict to smaller neighborhoods.
 Hence one can focus on polynomials instead of their germs.}
In particular, the spectrum (a discrete set of invariant associated to the Hodge structure of the vanishing cohomology, see Appendix \ref{secpreliminariesonsingularities}, p.\pageref{spectrum}) of $f_{a}:X(a)=\mathbb{C}^{n+1}\rightarrow \mathbb{C}$ is independent of $a$.

\subsubsection{Mixed Hodge structure on vanishing cohomology}
\label{secmixedHodgestructuresonvanishingcohomology}
Consider a quasi-homogeneous polynomial
  $f_{a}:X(a)=\mathbb{C}^{n+1}\rightarrow \mathbb{C}$ with an isolated
  critical point at \(0\). Denote by \(X(a)_{\infty}\) its Milnor fiber, see
\eqref{eqngeneralsingularity}.
The vanishing cohomology $H(a):=H^{n}(X(a)_{\infty},\mathbb{C})$ has an eigenspace decomposition\footnote{The last isomorphism
  is the isomorphism $\psi(a)$
  given in \eqref{eqnvanishingcohomologyasfiber}
 in Appendix  \ref{secpreliminariesonsingularities}, whose notation we omit. Structures defined on $C_{\alpha}$, such as the Hodge filtration,
are usually transferred to $H_{\lambda}$ via this isomorphism without explicitly mentioning.}
\begin{equation}\label{eqnpsiiso}
H(a)=\bigoplus_{\lambda:|\lambda|=1}H_{\lambda}(a)
\cong
\bigoplus_{\alpha:-1<\alpha\leq 0}C_{\alpha}(a)
\,,
\end{equation}
where $\lambda=e^{-2\pi i \alpha}$ is the eigenvalue of the monodromy $T=T_{ss}T_{u}=T_{ss}e^{2\pi i N}$
of the singularity $f_{a}$
around its singular fiber $f_{a}^{-1}(0)$, and $C_{\alpha}(a)\cong H_{\lambda}(a)$ is the corresponding generalized eigenspace
of the operator $t\partial_t$ (see \eqref{eqnFpC}).
In particular, one has the following weight filtration, as part of the data of the mixed Hodge structure $(W_{\bullet}(a), F^{\bullet}(a))$ on $H(a)$
\begin{equation*}
\cdots \subseteq W_{n-1}H(a)=0\subseteq W_{n} H(a)=\bigoplus_{\lambda\neq 1}H_{\lambda}(a)\subseteq W_{n+1} H(a)=H(a)\subseteq\cdots
\end{equation*}
This gives rise to the exact sequence
\begin{equation}\label{eqngrWofvanishingcohomology}
0\rightarrow \mathrm{gr}^{W}_{n}H(a)=\bigoplus_{\lambda\neq 1}H_{\lambda}(a)\rightarrow H(a)\rightarrow \mathrm{gr}^{W}_{n+1}H(a)=\bigoplus_{\lambda= 1}H_{\lambda}(a)\rightarrow 0
\end{equation}

The projective CY variety $Q_{a}$  carries a Hodge structure.
It is a standard fact that $\mathrm{gr}^{W}_{n+1}H(a)$ is isomorphic to the Hodge structure
on the primitive part $H^{n-1}_{\mathrm{prim}}(Q_a)$ of $H^{n-1}(Q_a)$ via the Poincar\'e residue map.
In fact, by the quasi-homogeneity one has (see e.g., \cite[Chapter 3.1.13]{Dimca:1992}) $H(a)\cong H^{n}(U_{a})$, where
$U_a$ is the variety given by $f_{a}=1$ in $\mathbb{C}^{n+1}$.
This variety $U_{a}$ can be compactified to
a projective variety $Y_{a}$, such that $U_a=Y_{a}-Q_{a}$ is realized as a hypersurface complement. The Gysin sequence then gives a sequence on mixed Hodge structures
\begin{equation}\label{eqnGysin}
\cdots\rightarrow H^{n}(Y_{a})\rightarrow H^{n}(U_{a})\xrightarrow{\mathrm{res}}H^{n-1}(Q_{a})\otimes \mathbb{C}(-1)\rightarrow H^{n+1}(Y_{a})\rightarrow\cdots
\,,
\end{equation}
where $\mathrm{res}$ is the Poincar\'e residue map and $\mathbb{C}(-1)$ is the Tate structure of weight $2$. Here the weight filtration on $H^{n}(U_{a})$
is given by
\begin{equation}\label{eqnMHSonU}
W_{n}H^{n}(U_{a})=\mathrm{im}\left(H^{n}(Y_{a})\rightarrow H^{n}(U_{a})\right)\subseteq W_{n+1}H^{n}(U_{a})=H^{n}(U_{a})\,.
\end{equation}
Combining \eqref{eqngrWofvanishingcohomology}, \eqref{eqnGysin} and \eqref{eqnMHSonU} and the isomorphism $H(a)\cong H^{n}(U_{a})$, we obtain
\begin{equation}\label{eqnLGCY}
H(a)\rightarrow \mathrm{gr}^{W}_{n+1}H(a)\xrightarrow{\mathrm{res}}H^{n-1}_{\mathrm{prim}}(Q_a)\otimes \mathbb{C}(-1)\,.
\end{equation}
This map provides a realization \cite{Cecotti:1991} of  the LG/CY correspondence on the level of state spaces using
the language of mixed Hodge structures.

\subsubsection{Twisted sectors in vanishing cohomology}

Due to the quasi-homogeneity of $f_a$, the monodromy $T$ is semi-simple.
This makes the mixed Hodge structure $(W_{\bullet}(a), F^{\bullet}(a))$ on the vanishing cohomology $H(a)$ very concrete.

\begin{dfn}\label{dfntwistedsectormonomials}
Denote $\bold{m}=(m_{i})_{i=0}^{n}$ and $\bold{z}^{\bold{m}}=
\prod_{i=0}^{n}z_{i}^{m_{i}}$.
Let
\begin{equation}\label{eqnmgrading}
\beta(\bold{m})=\sum_{i=0}^{n} m_{i}q_{i}\,,
\end{equation}
and $\alpha(\bold{m})$  be the unique integral shift of $\beta(\bold{m})$ that is
valued in $(-1,0]$.
\end{dfn}
A basis of
$(f_{a})_{*}\Omega_{X(a)/\mathbb{C}}^{n+1}$
is represented by the differential forms
\begin{equation}\label{eqnmbasis}
\bold{z}^{\bold{m}}\Omega
\,, \quad \Omega:=\wedge_{i=0}^{n}dz_{i}\,,\quad
\bold{z}^{\bold{m}}\in \mathrm{Jac}\,(f_{0})\,.
\end{equation}
The classes $[\bold{z}^{\bold{m}}\Omega]$ of these differential forms generate the Brieskorn lattice
$ \mathcal{H}^{(0)}(a)$, which is an extension of the relative de Rham cohomology sheaf $\mathbb{R}^{n}{f_{a}}_{*}\Omega^{\bullet}_{X(a)/\mathbb{C}}$ (under an isomorphism provided by the map $df_{a}\wedge$).
Their Gelfand-Leray residues, called
geometric sections $s[\bold{z}^{\bold{m}}\Omega]$, define classes $\zeta[\bold{z}^{\bold{m}}\Omega]$ in the
vanishing cohomology through their leading parts.
These geometric sections are
eigenvectors of the operator $t\partial_{t}$ with eigenvalue (i.e., spectral number) given by $\beta=\beta(\bold{m})$. That is,
\begin{equation}\label{eqnpartialttwistedsectors}
\left(t\partial_{t} -\beta(\bold{m})\right)\, s [\bold{z}^{\bold{m}}\Omega]=0\,.
\end{equation}
 The Hodge filtration $F^{\bullet}(a)$ on the vanishing cohomology $H(a)$ is conveniently described using the geometric sections as follows.
Recall that the Hodge filtration $F^{\bullet}(a)$ on $H(a)$
is defined individually on each $H_{\lambda}(a)\cong C_{\alpha}(a)$.
Let
\begin{equation}\label{eqnbetap}
\beta=n-p+\alpha\,,
\end{equation}
where
$p$ is the degree in Hodge filtration and
$\alpha\in (-1,0]$.
Then
the space
$F^{p}H_{\lambda}(a)\cong F^{p}C_{\alpha}(a)$ is generated by
the classes $\zeta[\bold{z}^{\bold{m}}\Omega]$
obtained from
the leading parts of the geometric sections $s[\bold{z}^{\bold{m}}\Omega]$
such that $\beta(\bold{m}) \leq \beta$ in \eqref{eqnbetap}.
Due to the quasi-homogeneity, the leading part $\zeta[\bold{z}^{\bold{m}}\Omega]\in F^{p}H_{\lambda}(a)$ of
the geometric section $s[\bold{z}^{\bold{m}}\Omega]$ is essentially the geometric section $s[\bold{z}^{\bold{m}}\Omega]$  itself.
See \eqref{eqnzetaclass} in Appendix  \ref{secpreliminariesonsingularities} for more detailed explanations.

One then has
a decomposition for $\mathrm{gr}_{F}^{\bullet}H_{\lambda}(a)$  (recall $\lambda=e^{-2\pi i \beta}, n+\alpha-\beta=p$)
\begin{equation}\label{eqndecompositionbydegree}
   H_{\lambda,\beta}(a):=\mathrm{gr}_{F}^{n+\alpha-\beta} H_{\lambda}(a)\cong
   \mathrm{gr}_{F}^{n+\alpha-\beta}  C_{\alpha}(a)
   \,,\quad  \mathrm{gr}_{F}^{\bullet}H_{\lambda}(a)=\bigoplus_{\beta:\, e^{-2\pi i \beta}=\lambda}H_{\lambda,\beta}(a)
   \,,
 \end{equation}
 where elements in $H_{\lambda,\beta}(a)$ are given by the classes $\zeta[\bold{z}^{\bold{m}}\Omega]$ with $\beta=\beta(\bold{m})$.
 For a quasi-homogeneous singularity, the spectrum $\beta$ together with the dimensions $n_{\beta}:=\mathrm{dim}H_{\lambda,\beta}$ can be easily worked out
 following the methods in \cite{Kulikov:1998}.\\

 The structures of the eigenspace decomposition \eqref{eqnpsiiso} and the associated graded ring \eqref{eqndecompositionbydegree}  with respect to the Hodge filtration lead to the notion of twisted sectors.
 \begin{dfn}\label{dfnvarioussectors}
   For a quasi-homogeneous singularity $f_{a}$, set $\beta=n-p+\alpha$ with $\alpha\in (-1,0]$ as in \eqref{eqnbetap}.
   An element in
   $H_{\lambda,\beta}(a)$
   is called a twisted sector.\footnote{In the physics literature, usually only
     sectors with $\lambda\neq 1$ are called twisted sectors. }
   It is called a
   \begin{itemize}
   \item relevant sector, if $0\leq \beta<1$.
   \item marginal sector, if $\beta=1$.
   \item irrelevant sector, if $\beta>1$.
   \end{itemize}
   The subspace of sectors with $\lambda=1$, that is  $\mathrm{gr}_{F}^{\bullet} H_{1}(a)$,
   is mapped to the space  $H^{n-1}_{\mathrm{prim}}(Q_{a})$
 under the Poincar\'e residue map \eqref{eqnLGCY} and hence is called the Calabi-Yau sector.
 In particular, the one with $\beta=0$, which spans $F^{n}H_{1}(a)$, is mapped to $F^{n-1}H^{n-1}_{\mathrm{prim}}(Q_{a})$ that is the space of Calabi-Yau forms on $Q_{a}$.
\end{dfn}

As explained above, the classes $\zeta[\bold{z}^{\bold{m}}\Omega]$  arising from the leading parts of the geometric sections $s[\bold{z}^{\bold{m}}\Omega]$ provide a basis for the space  $\bigoplus_{\lambda}\mathrm{gr}_{F}^{\bullet} H_{\lambda}(a)$.
The data $\bold{m}$ itself
actually arises from a group action on the ambient space $\mathbb{C}^{n+1}$ as follows.
Introduce the action of $G=(\mathbb{Z}/(n+1)\mathbb{Z})^{n+1}$ on the coordinate ring $\mathbb{C}[z_0,z_1,\cdots, z_n]$
\begin{equation}
g=(k_0,k_1,\cdots, k_{n})\in G: \,(z_0,z_1,\cdots ,z_{n})\mapsto (\xi_{n+1}^{k_0} z_0, \xi_{n+1}^{k_1} z_1, \cdots, \xi_{n+1}^{k_n}  z_n )\,,
\end{equation}
with $\xi_{n+1}:=\exp({2\pi i/( n+1)})$.
Then the eigenspaces of the $G$-action are exactly given by
 the monomials $\bold{z}^{\bold{m}}$.
It follows that
the space $H_{\lambda,\beta}(a)$ of twisted sectors can be further decomposed into the corresponding $G$-eigenspaces
\begin{equation}\label{eqncharacterdecomposition}
H_{\lambda,\beta}(a)=\bigoplus_{\bold{m}: \beta(\bold{m})=\beta}H_{\lambda, \beta, \bold{m}}(a)\,,\quad
\mathrm{gr}_{F}^{\bullet}H_{\lambda}(a)=\bigoplus_{\beta}\bigoplus_{\bold{m}: \beta(\bold{m})=\lambda}H_{\lambda, \beta, \bold{m}}(a)
=\bigoplus_{\bold{m}}H_{\lambda, \beta, \bold{m}}(a)\,.
\end{equation}

All of the above constructions for $f_{a}:X(a)=\mathbb{C}^{n+1}\rightarrow \mathbb{C}$ glue to constructions for
\begin{equation*}
\varphi=(F,q): \quad X=\mathbb{C}^{n+1}\times S\rightarrow \mathbb{C}\times S\,,
\end{equation*}
where  $F$ is the deformation in \eqref{eqnFdfn} and $q:X\rightarrow S$ is the projection.
In particular, the vanishing cohomology bundle is a local system over $ S$.
See \cite{Kulikov:1998} and references therein for details.
We use the same notation without  the '$(a)$' part  to denote the latter, whose restrictions at $a\in S$ recover
the former.

\subsection{Ordinary differential equations for geometric sections}
\label{eqndifferentialequationsforsectors}

From now on, we concentrate on the Dwork family \eqref{eqnDworksingularity} in the setting of
\eqref{eqnFdfn}.
The basis in \eqref{eqnmbasis} is now given by
\begin{equation}
\bold{z}^{\bold{m}}\Omega
\,, \quad \Omega:=\wedge_{i=0}^{n}dz_{i}\,,\quad
0\leq m_{i}\leq n-1\,.
\end{equation}
Our goal is to study differential equations satisfied by the section $\zeta[\bold{z}^{\bold{m}}\Omega]$ of the bundle $H_{\lambda,\beta(\bold{m})}\rightarrow S$ of twisted sectors.
Due to the constancy \eqref{eqnpartialttwistedsectors} under the $t\partial_{t}$-structure of the geometric sections
$s[\bold{z}^{\bold{m}}\Omega]$, this is equivalent to studying the $a, \partial_{a}$
actions on the latter---which is easier using the tool of differential forms. For this reason,
hereafter we shall not distinguish $[\bold{z}^{\bold{m}}\Omega]$ from the geometric section $s[\bold{z}^{\bold{m}}\Omega]$ notationally,
and properties on them will be transferred to those on the vanishing cohomology classes $\zeta[\bold{z}^{\bold{m}}\Omega]$ without explicitly mentioning.
For example, the following relation obtained from direct computations \cite{Kulikov:1998}
\begin{equation}\label{eqnGMongeometricsections}
\partial_{a}\, s[\bold{z}^{\bold{m}}\Omega]=-\partial_{t}\,s[\partial_{a}F\cdot \bold{z}^\bold{m} \Omega]\,
\end{equation}
gets translated into a relation\footnote{Note the $\partial_{t}^{n-p}$ operator part in the definition of $F^{p}C_{\alpha}$ in \eqref{eqnFpC}  of Appendix  \ref{secpreliminariesonsingularities}.} on the corresponding leading parts $\zeta[\bold{z}^{\bold{m}}\Omega]$
in vanishing cohomology
\begin{equation}\label{eqnGMonvanishingcohomologyclasses}
\partial_{a}\, \zeta[\bold{z}^{\bold{m}}\Omega]=-\,\zeta[\partial_{a}F\cdot \bold{z}^\bold{m} \Omega]\,.
\end{equation}

For simplicity we have decided to restrict ourselves to the case $e_{1}(z)=\prod_{i=0}^{n}z_{i}$ in the setting of
\eqref{eqnFdfn}.
Discussions for the more general cases
are similar.

Direct computations following the methods in \cite{Kulikov:1998} give the following result.
\begin{lem}\label{lemPicardFuchs}
Consider the Fermat polynomial $f= \sum_{i=0}^{n}z_{i}^{n+1}$.
Along the deformation $e_{1}(z)=\prod_{i=0}^{n}z_{i}$ in \eqref{eqnoneparametermarginaldeformation}, one has
\begin{equation}
S=\mathbb{C}-\{ a^{n+1}-(n+1)^{n+1}=0\}\,.
\end{equation}
The geometric section
$s[\bold{z}^{\bold{m}}\Omega]\in \mathcal{H}^{(0)}$  satisfies the differential equation
\begin{equation}\label{eqnPicardFuchs}
\left(\prod_{i=0}^{n}(\theta_{a}+i-n)
-{a^{n+1}\over (n+1)^{n+1}} \prod_{i=0}^{n}(\theta_a+m_{i}+1)\right)s[\bold{z}^{\bold{m}}\Omega]=0\,,\quad
\theta_{a}:=a{\partial\over \partial a}\,.
\end{equation}
This differential equation can be possibly reduced to a lower order equation obtained by factoring out one with highest possible order from the left,
therefore with reduced order
\begin{equation}\label{eqnreducedordertwistedsectorequation}
n+1-\mathrm{cardinality}(\{m_{i}+1\}_{i=0}^{n})\,.
\end{equation}
In particular, the differential equation is independent of the
$\mathfrak{S}_{n+1}$ permutations on $\bold{m}$.

\end{lem}
\begin{proof}
The first assertion about the discriminant locus follows from a direct computation
of the possible singularity locus of $f_{a}$ by using the Jacobian criterion.
In particular, the Milnor number is $\mu=\dim \mathrm{Jac}(f)=n^{n+1}$.
The assertion regarding the differential operators follow from straightforward but tedious computations following the methods in \cite{Kulikov:1998}.
\end{proof}

To reduce the number of singularities of the differential equations, we introduce a more convenient coordinate
\begin{equation}\label{eqnbasechange}
b={a^{n+1}\over (n+1)^{n+1}}\in\mathcal{M}=\mathbb{C}-\{0,1\}\,.
\end{equation}
As we shall see later, this change of coordinate will be reflected in the orbifold structure that one puts on $\mathcal{M}$.
Similar to Meijer's $G$-function \cite{Erdelyi:1981}, we introduce the series
\begin{equation*}
_{r}G_{s}(\alpha_{1},\cdots, \alpha_{r};\beta_{1},\cdots ,\beta_{s}; z)=\sum_{k=0}^{\infty}{\prod_{i=1}^{r}(\alpha_{i})_{k}\over \prod_{j=1}^{s}(\beta_{j})_{k}}z^{k}\,,\quad
(\alpha)_{k}={\Gamma(\alpha+k)\over \Gamma(\alpha)}\,.
\end{equation*}
Besides the obvious relation
\begin{equation*}
_{r}G_{s}(\alpha, \cdots; \alpha,\cdots;z)=\,_{r-1}G_{s-1}(\cdots; \cdots;z)\,,
\end{equation*}
this series satisfies the differential equation (if some $\beta_{j}=1$)
\begin{equation}\label{eqnGfunction}
\left(\prod_{j=1}^{s}(\theta_{z}+\beta_j-1)
-z \prod_{i=1}^{r}(\theta_z+\alpha_i)\right)\,_{r}G_{s}(\alpha_{1},\cdots, \alpha_{r};\beta_{1},\cdots ,\beta_{s}; z)=0\,,\quad
\theta_{z}:=z{\partial\over \partial z}\,.
\end{equation}

We can now see that  period integrals of the geometric section $s[\bold{z}^{\bold{m}}\Omega]$, obtained by integrating it on locally constant homology classes
in the homological Milnor fibration (see \eqref{eqnzetaclass} in Appendix  \ref{secpreliminariesonsingularities}),
can be expressed in terms of the above $_{r}G_{s}$ functions.
\begin{cor}\label{corperiodintegrals}
Consider the Dwork family of polynomial singularities \eqref{eqnDworksingularity}.
Period integrals of the geometric section  $s[\bold{z}^{\bold{m}}\Omega]$ are $\mathbb{C}$-linear combinations of
\begin{equation}\label{eqnperiodintegralofgeometricsection}
a^{\delta}\cdot\,_{n+1}G_{n+1}({\delta+m_0+1\over n+1},\cdots, {\delta+m_n+1\over n+1};{\delta+1\over n+1},\cdots,{\delta+n+1\over n+1}; b)\,,
\end{equation}
where
$\delta\in\{0,1,\cdots,n\}$ is such that $\delta+m_i+1\neq 0\in \mathbb{Z}/(n+1)\mathbb{Z}\,,\forall\, i=0,1,\cdots,n$.
\end{cor}
\begin{proof}
From \eqref{eqnGfunction} we see that  functions in \eqref{eqnperiodintegralofgeometricsection} satisfy the differential equations \eqref{eqnPicardFuchs} in Lemma \ref{lemPicardFuchs}.
From their asymptotic behaviors near $a=0$, we see that
these solutions are linearly independent. Hence this set of solutions span a vector space of dimension
$$\mathrm{cardinality} \left(\{\delta~|~\delta+m_i+1\neq 0\in \mathbb{Z}/(n+1)\mathbb{Z}\,,\forall i=0,1,\cdots,n\}\right)\,.$$
This is exactly the order of the differential equation with reduced order given in Lemma \ref{lemPicardFuchs}.
The desired claim then follows.
\end{proof}

\subsubsection{Real structure and oscillating integrals}\label{secintegralstructure}

Recall that for the cohomology of the twisted de Rham complex $(\Omega^{\bullet}_{X(a)}, d-df_{a}\wedge)$ one has
\begin{equation*}
 \mathbb{H}^{n+1}(\Omega^{\bullet}_{X(a)}, d-df_{a}\wedge)\cong H^{n+1}(\Gamma(X(a),\Omega_{X(a)}^{\bullet}),d-df_{a}\wedge)\,.
 \end{equation*}
The real structure of the latter is compatible with the one on the homology group of Lefschetz thimbles
via a comparison theorem, which in the present case is provided by the oscillating integrals\footnote{The association of oscillating integrals is not canonical. For example, one could have taken the differential to be $d-{1\over \hbar}df_{a}\wedge $
and correspondingly considered the oscillating integrals of $e^{-{1\over \hbar}f_{a}}\bold{z}^{\bold{m}}\Omega$, with $\hbar\in\mathbb{R}_{>0}$. This would also change the isomorphism
\eqref{eqncomparison} below.}
\begin{equation*}
\langle \bold{z}^{\bold{m}}\Omega,\gamma \rangle:=\int_{\gamma}e^{-f_{a}}\bold{z}^{\bold{m}}\Omega\,,\quad \gamma\in H_{n+1}(\mathbb{C}^{n+1}, (\mathrm{Re}\,f_{a})^{-1}(+\infty);\mathbb{Z})\,.
\end{equation*}
It turns out that the pairing $\langle-,-\rangle$ defined above gives a perfect pairing.
See \cite{Malgrange:1974, Pham:1985, Arnold:1988} for details, see also \cite{Gross:2011} for a nice exposition.
On the other hand, from the exact sequence for singular homology and cohomology groups with complex coefficients, we have
\begin{eqnarray}\label{eqncomparison}
H^{n}(X(a)_{\infty})&\cong& (H_{n}(X(a)_{\infty}))^{*}\nonumber\\
&\cong &( H_{n+1}(\mathbb{C}^{n+1}, (\mathrm{Re}\, f_{a})^{-1}(+\infty);\mathbb{C}))^{*}\nonumber\\
&\cong& H^{n+1}(\Gamma(X(a),\Omega_{X(a)}^{\bullet}),d-df_{a}\wedge)\,.
\end{eqnarray}
Hence the real structure on the vanishing cohomology
$H^{n}(X(a)_{\infty})$ is compatible with that on the twisted cohomology group.
Explicitly, the isomorphism \eqref{eqncomparison} maps the class $\zeta[\bold{z}^{\bold{m}}\Omega](a)\in F^{p}H_{\lambda}(a)$, arising from the leading part of
the geometric section $s[\bold{z}^{\bold{m}}\Omega](a)\in \Gamma(X(a),\mathcal{H}^{(0)})$, to the class $[\bold{z}^{\bold{m}}\Omega]$ of $\bold{z}^{\bold{m}}\Omega\in\Gamma(X(a),\Omega_{X(a)}^{n+1})
$.

The homology group $H_{n+1}(\mathbb{C}^{n+1}, (\mathrm{Re}\, f_{a})^{-1}( +\infty);\mathbb{Z})$ has rank $n^{n+1}$ and is generated by cycle classes of the form
\begin{equation*}
\gamma=\gamma_{0}\times \cdots\times \gamma_n\,, \quad \gamma_{i}\in H_{1}(\mathbb{C}, \mathrm{Re}\, z_{i}=+\infty;\mathbb{Z})\,,~i=0,1,\cdots, n\,.
\end{equation*}
The construction of the cycles is as follows. 
Let 
$c_{i}$ be the chain connecting $z_i=0$ and $z_{i}=+\infty$ along the real axis in $\mathbb{C}$,
and $\xi_{n+1}c_{i}$ be the one obtained from the action of $\xi_{n+1}\in \mathbb{C}^{*}$ on $c_{i}$,
with $\xi_{n+1}=e^{2\pi i\over n+1}$.
Then each
$H_{1}(\mathbb{C}, \mathrm{Re}\, z_{i}=+\infty;\mathbb{Z})$
has rank $n$ and is generated by
$c_{i}-\xi_{n+1}^{h_{i}}c_{i}$ with $ h_{i}\in \{1,\cdots, n\}$.
Therefore, $H_{n+1}(\mathbb{C}^{n+1}, (\mathrm{Re}\, f_{a})^{-1}(+\infty);\mathbb{Z})$ is
generated by
\begin{equation*}
\gamma_{\bold{h}}=\prod_{i=0}^{n}(1-\xi_{n+1}^{h_{i}})c_{i}\,,\quad
\bold{h}=(h_{0},\cdots, h_{n})\in \{1,2,\cdots,n\}^{\oplus (n+1)} \,.
\end{equation*}
The corresponding period integral can then be evaluated directly to be
\begin{equation}\label{eqnoscilatingintegrals}
\langle \bold{z}^{\bold{m}}\Omega, \gamma_{\bold{h}}
\rangle
=\int_{\gamma_{\bold{h}}}e^{-f_{a}}\bold{z}^{\bold{m}}\Omega
=\sum_{d\geq 0}\prod_{i=0}^{n}(1-\xi_{n+1}^{h_{i}(d+m_{i}+1)})\Gamma({d+m_{i}+1\over n+1}) {a^{d}\over (n+1)^{n+1} d!}\,.
\end{equation}

Let
\begin{equation*}
d=(n+1)\ell+\delta\,,\quad \ell\in \mathbb{N}\,,~\delta\in \{0,1,\cdots,n\}\,.
\end{equation*}
In order for the above period integral \eqref{eqnoscilatingintegrals} to be nonzero, one must have
\begin{equation*}
\delta+m_{i}+1\neq 0\in \mathbb{Z}/(n+1)\mathbb{Z}\,,\quad \forall\,i=0,1,\cdots, n\,.
\end{equation*}
The resulting expression of \eqref{eqnoscilatingintegrals} is then
\begin{equation}\label{eqnoscilatingintegralsintermsofsumoverdelta}
	\langle \bold{z}^{\bold{m}}\Omega, \gamma_{\bold{h}}
	\rangle=
\sum_{\delta}
\prod_{i=0}^{n}(1-\xi_{n+1}^{h_{i}(\delta+m_{i}+1)})\cdot
a^{\delta}\sum_{\ell\geq 0} \prod_{i=0}^{n}\Gamma({\ell}+{\delta+m_{i}+1\over n+1}){a^{(n+1)\ell}\over ((n+1)\ell+\delta)! (n+1)^{n+1}}
\,.
\end{equation}
Using the Gauss multiplication formula for Gamma functions, one has
\begin{eqnarray}\label{eqnoscilatingintegralintermsofhypergeometric}
&&\prod_{i=0}^{n}(1-\xi_{n+1}^{h_{i}(\delta+m_{i}+1)})\cdot a^{\delta}\sum_{\ell\geq 0} \prod_{i=0}^{n}\Gamma({\ell}+{\delta+m_{i}+1\over n+1}){a^{(n+1)\ell}\over ((n+1)\ell+\delta)! (n+1)^{n+1}}\\ \nonumber
& =& \prod_{i=0}^{n}(1-\xi_{n+1}^{h_{i}(\delta+m_{i}+1)})\cdot
 (2\pi )^{n\over 2}(n+1)^{-{1\over 2}-\delta-(n+1)} {\prod_{i=0}^{n} \Gamma({\delta+m_{i}+1\over n+1})
 \over \prod_{i=0}^{n} \Gamma({\delta+i+1\over n+1})}\cdot
  \text{expression in}~\eqref{eqnperiodintegralofgeometricsection}
 \,.
 \end{eqnarray}
Denote the linear dual of  $\gamma_{\bold{h}}$ by $\check{\gamma}_{\bold{h}}$,
we have
\[
[ \bold{z}^{\bold{m}}\Omega]=\sum_{\bold{h}}
\langle \bold{z}^{\bold{m}}\Omega, \gamma_{\bold{h}}\rangle\check{\gamma}_{\bold{h}}\,.
\]
We have therefore identified the real (and even integral) structure on $H(a)$ through computing oscillating integrals $\langle \bold{z}^{\bold{m}}\Omega, \gamma_{\bold{h}}
\rangle$ in \eqref{eqnoscilatingintegrals}.
See Section \ref{secexamples} for examples.

\subsection{Automorphic bundles corresponding to D-modules}
\label{secRHcorrespondence}

For any $\bold{m}$, denote
\begin{equation*}
\bold{m}+\ell\bold{1}:=(m_0+\ell,m_1+\ell,\cdots,m_{n}+\ell)\,,\quad \ell\in \mathbb{N}\,.
\end{equation*}
According to Lemma \ref{lemPicardFuchs} in Section \ref{eqndifferentialequationsforsectors},
for each $\bold{m}$ the vanishing cohomology classes $\{\zeta[\bold{z}^{\bold{m}+\ell\bold{1}}\Omega]\}_{\ell\in\mathbb{N}}$  span
a regular holonomic $D_{S}$-module
governed by  the relation
\eqref{eqnGMonvanishingcohomologyclasses}.
We denote the corresponding flat vector bundle over $S$
by $D_{\bold{m}}$.
In particular,
the relation \eqref{eqnGMonvanishingcohomologyclasses} provides an isomorphism
\begin{equation}\label{eqnshiftiso}
D_{\bold{m}}\cong D_{\bold{m}+\bold{1}}\,.
\end{equation}
To be more concrete, denote the reduced differential operator in Lemma \ref{lemPicardFuchs}
 by $\mathcal{D}_{\bold{m}}$, then
\begin{equation}\label{eqncommutingdifferentialoperators}
\mathcal{D}_{\bold{m}+\bold{1}}\circ\partial_a={1\over a}(\theta_{a}-n-1)\circ \mathcal{D}_{\bold{m}}\,.
\end{equation}
The Hodge filtration on $H_{\lambda}$
induces a filtration on $D_{\bold{m}}$
which we still denote by $F^{\bullet}$,
with the grading of the filtration conveniently recorded  in the value $\beta(\bold{m})$ according to \eqref{eqnbetap}.
As explained at the beginning of Section \ref{eqndifferentialequationsforsectors}, we can and shall regard $D_{\bold{m}}$ as
a $D_{S}$-module spanned by the geometric sections $s[\bold{z}^{\bold{m}}\Omega]$, as long as only the
$D_{S}$-module structure on the vanishing cohomology classes $\zeta[\bold{z}^{\bold{m}}\Omega]$ is concerned.\\

The D-module $D_{\bold{m}}$
descends to  $\mathcal{M}=\mathbb{P}^{1}-\{0,1,\infty\}$ along the map $S\rightarrow \mathcal{M}$ given in \eqref{eqnbasechange}.
Computationally, the resulting differential operator is obtained by substituting
$\theta_a=(n+1)\theta_b$
in \eqref{eqnPicardFuchs} in
Lemma \ref{lemPicardFuchs}
\begin{equation}\label{eqnPicardFuchsbasechanged}
\left(\prod_{i=0}^{n}(\theta_{b}+{i-n\over n+1})
-b \prod_{i=0}^{n}(\theta_b+{m_{i}+1\over n+1})\right)s[\bold{z}^{\bold{m}}\Omega]=0\,,\quad
\theta_{b}:=b{\partial\over \partial b}\,.
\end{equation}

We now prove a main result Theorem \ref{thmautomorphicform} in our work.
To proceed, we first recall some notation on orbifold lines.
Consider the orbifold
\begin{equation}
T_{\ell_{0},\ell_{1},\ell_{\infty}}:=\Gamma_{\ell_{0},\ell_{1},\ell_{\infty}}
\backslash \mathcal{H}\,,
\end{equation}
where $\mathcal{H}$ is the upper-half plane, and
$\Gamma_{\ell_{0},\ell_{1},\ell_{\infty}}<\mathrm{PSL}_{2}(\mathbb{R})$ is the triangular group
that admits the following presentations
\begin{eqnarray}
\Gamma_{\ell_{0},\ell_{1},\ell_{\infty}}&=&
\langle
\sigma_{0}, \sigma_{1},\sigma_{\infty}
\rangle/\langle \sigma_{0}^{\ell_0}=\sigma_{1}^{\ell_{1}}=\sigma_{\infty}^{\ell_{\infty}}
=\sigma_{\infty}\sigma_{1}\sigma_{0}=1\rangle\nonumber\\
&\cong&
\langle
\sigma_{0}, \sigma_{1}
\rangle/\langle \sigma_{0}^{\ell_0}=\sigma_{1}^{\ell_{1}}=(\sigma_{1}\sigma_{0})^{\ell_{\infty}}=1\rangle\,.
\label{eqnpresentationtriangulargroup}
\end{eqnarray}
In the above we have used the convention that
 $\sigma^{\ell}=1$ if $\ell=\infty$.
By tautology one has
\begin{equation}
\pi_{1}^{\mathrm{orb}}(T_{\ell_0,\ell_1,\ell_\infty})=\Gamma_{\ell_{0},\ell_{1},\ell_{\infty}}\,.
\end{equation}
We then have the following result.

\begin{thm}\label{thmautomorphicform}
Consider the Dwork family of polynomial singularities \eqref{eqnDworksingularity}, together with the map $S\rightarrow \mathcal{M}$ given in \eqref{eqnbasechange}.
For each $\bold{m}$, there exists a triple $\bold{\ell}(\bold{m})=(\ell_{0}(\bold{m})=n+1, \ell_{1}(\bold{m}),\ell_{\infty}(\bold{m}))$
such that
the corresponding flat bundle $D_{\bold{m}}$ over $\mathcal{M}$
is an automorphic bundle
 and thus
the geometric section
$s[\bold{z}^{\bold{m}}\Omega]$ is the component of an automorphic form,
for the triangular group $\Gamma_{\ell_0(\bold{m}),\ell_1(\bold{m}),\ell_\infty(\bold{m})}<\mathrm{PSL}_{2}(\mathbb{R})$.
In particular, all of the geometric sections are components of automorphic forms for $\Gamma_{n+1,\infty,\infty}$.
\end{thm}

\begin{proof}
Let $\ell_{0}(\bold{m}),\ell_{1}(\bold{m}),\ell_{\infty}(\bold{m})$ be the orders of the local monodromies of $D_{\bold{m}}$ respectively.
It follows that under the Riemann-Hilbert correspondence,
$D_{\bold{m}}$ corresponds to a monodromy representation of
$\Gamma_{\ell_0(\bold{m}),\ell_1(\bold{m}),\ell_\infty(\bold{m})}$.
Consider the indicial roots at the singular point $b=0$ of \eqref{eqnPicardFuchsbasechanged}. It is direct to check that they are
all different, and thus
the indicial differences at $b=0$ are multiples of $1/(n+1)$.
This tells that $\ell_0(\bold{m})$ is a factor of $(n+1)$.
By Definition \ref{defautomorphicbundle},
the bundle $D_{\bold{m}}$ is an automorphic bundle for the triangular group $\Gamma_{\ell_0(\bold{m}),\ell_1(\bold{m}),\ell_\infty(\bold{m})}$
and the section $s[\bold{z}^{\bold{m}}\Omega]$ is a component of an automorphic form for this group.
The last assertion follows from the fact that
any representation of $\Gamma_{\ell_{0}, \ell_1,\ell_\infty}$ with $\ell_0\mid (n+1)$
is a representation of $\Gamma_{n+1,\infty,\infty}$ that is induced by a quotient homomorphism
$\Gamma_{n+1,\infty,\infty}\rightarrow \Gamma_{\ell_{0}, \ell_1,\ell_\infty}$.
\end{proof}

For each $\bold{m}$, let $\overline{\mathcal{M}}_{\bold{m}}$ be the orbifold obtained from $\mathcal{M}$ by filling in the elliptic points (i.e., those with finite-order monodromies).
As a consequence of Theorem \ref{thmautomorphicform}, the Riemann-Hilbert correspondence provides
$\overline{\mathcal{M}}_{\bold{m}}$ with the orbifold structure $T_{\ell_0(\bold{m}),\ell_1(\bold{m}),\ell_\infty(\bold{m})}$
with structure map
$\mathcal{H}\rightarrow T_{\ell_0(\bold{m}),\ell_1(\bold{m}),\ell_\infty(\bold{m})} $.
The possible multi-valuedness of $s[\bold{z}^{\bold{m}}\Omega]$
reflects the fact that it is a section of an orbi-bundle over this orbifold.
In the following, to ease the notation we use the same notation $\mathcal{M}$ for $\overline{\mathcal{M}}_{\bold{m}}$ which should
cause no confusion from the surrounding context.

We remark that regarding a representation of $\Gamma_{\ell_0, \ell_1, \ell_\infty} $
as one for $\Gamma_{n+1, \infty, \infty}$ can lose significant information of the origin one.
Even worse, the classification of representations for $\Gamma_{n+1, \infty, \infty}$ and explicitly computing analytic expressions for the corresponding automorphic forms
can already be nontrivial.\\

The bundle $D_{\bold{m}}$ comes  with a filtration $F^{\bullet}D_{\bold{m}}$ induced from the Hodge filtration on vanishing cohomology.
These bundles naturally extend to locally free sheaves over the compactified orbifold $\overline{T}_{\bold{\ell}(\bold{m})}$
by Schmid's results \cite{Schmid:1973} on limiting Hodge filtrations.
We denote the extended bundles by $\{F^{p}\overline{D}_{\bold{m}}\}_{p}$.
In particular, the bundle $\overline{D}_{\bold{m}}$  coincides with Deligne's canonical extension in the context of ordinary differential equations
with regular singularities.
To be more precise, the extension data\footnote{Note that
in Deligne's canonical extension one uses the convention $\lambda=e^{-2\pi i \alpha}$
for the eigenvalue of $T$, while for the exponent $L$ one uses $T=e^{2\pi i L}.$
}  is given by logarithms of the local monodromies called exponents $L$
such that the real parts of eigenvalues of $L$ lie in $[0,1)$.
See
\cite{Saber2014:vectorvalued, Candelori:2017vector, Candelori:2019vector} for detailed discussions on extensions.

 Once the representations are obtained, then
 by definition sections of $\overline{D}_{\bold{m}} $,
 and in particular those of the sub-bundles $F^{p} \overline{D}_{\bold{m}}$,
  are automorphic forms with boundary conditions imposed at the elliptic points and cusps.

\begin{rem}
 Note that
  the Picard group and canonical bundle for both
 the orbifold and its compactification are
 uniquely determined by the data of signature and admit very concrete descriptions.
 For example, consider the compact orbifold line $C=\overline{T}_{k_0, k_1,k_{2}}$, denote by $p$ a generic point on $C$.
 Then the Picard group has the presentation
 \begin{equation*}
 \mathrm{Pic}(C)=\langle  \mathcal{O}_{C}(p_{i}),i=0,1,2~|~\mathcal{O}_{C}(p_{i})^{\otimes k_{i}}\cong \mathcal{O}_{C}(p), i=0,1,2\rangle \,.
 \end{equation*}
  The canonical sheaf, which is isomorphic to the relative dualizing sheaf, is given by
 \begin{equation*}
 \Omega_{C}^{1}=\mathcal{O}_{C}((3-1)p-\sum_{i=0}^{2}p_{i})\,.
 \end{equation*}

  Consider as a special case the orbifold $T_{\ell_0,\ell_1,\ell_\infty}=\Gamma_{\ell_0,\ell_1,\ell_\infty}\backslash\mathcal{H}$ corresponding to the triangular group
  $ \Gamma_{\ell_0,\ell_1,\ell_\infty}$ arising from Theorem \ref{thmautomorphicform}. The sheaf $\Omega^{1}_{T_{\ell_0,\ell_1,\ell_\infty}}$ is not necessarily flat, yet it is
  an automorphic bundle
whose automorphy factor can be easily worked out  using the Schwarzian uniformization.
  This canonical sheaf enters through the  connection on the flat vector bundles $D_{\bold{m}}$
\begin{equation*}
\nabla: D_{\bold{m}}\rightarrow D_{\bold{m}}\otimes \Omega^{1}_{T_{\ell_0,\ell_1,\ell_\infty}}\,.
\end{equation*}
The flatness of $\nabla$, which is equivalent to the D-module structure of $D_{\bold{m}}$, yields a differential structure
of automorphic sections of $D_{\bold{m}}$.
The transcendental degree of the differential field is then determined by the dimension of the differential Galois group
which is closely related to the monodromy group \cite{Beukers:1989}.
See Section \ref{secquinticexample} for related discussions.
The
connection $\nabla$ extends to a meromorphic connection with only logarithmic pole on the
extended bundle $\overline{D}_{\bold{m}}$
\begin{equation*}
\overline{\nabla}: \overline{D}_{\bold{m}}\rightarrow \overline{D}_{\bold{m}}\otimes \Omega^{1}_{\overline{T}_{\bold{\ell}(\bold{m})}}(\log \sum_{p_{i}}p_{i})\,,
\end{equation*}
where the summation in the logarithmic canonical sheaf is over the loci of cusps.
 \end{rem}

\subsection{Examples}
\label{secexamples}

In this section we discuss the details for the
$n=2,3$ cases.
We shall work out
the automorphy factors for the top pieces $F^{\mathrm{top}}D_{\bold{m}}$ using the
Schwarzian uniformization.
This shows that their corresponding sections are automorphic forms with weights given by $\mathrm{rank}\,D_{\bold{m}}-1$.
For the $D_{\bold{0}}$ sector, we shall show that these automorphic forms
actually reduce to elliptic modular forms.
We shall also discuss the $n=4$ case briefly and mention some interesting results
that have shown to be important in the studies of higher genus mirror symmetry for quintic threefolds.

\subsubsection{Fermat cubic}
\label{seccubicmarginalelliptic}

Consider the $n=2$ case.
The spectrum that encodes the information of the
weight filtration and Hodge filtration
is given by  Table \ref{table-spectrumcubic} below.
\begin{table}[h]
  \centering
\caption{Spectrum of the cubic polynomial singularity $f_{a}$\,.}
  \label{table-spectrumcubic}
  \renewcommand{\arraystretch}{1.5} 
 \begin{tabular}{c|ccccc}
 \hline
$\beta$ & $0$ & ${1\over 3}$ & ${2\over 3}$ & $1$\\
 \hline
$n_{\beta}$  & $1$ & $3$ & $3$ & $1$\\
\hline
\end{tabular}
\end{table}
As vector spaces one has
\begin{eqnarray*}
\mathrm{gr}^{W}_{3}H(a)&=& H_{1}(a)= \mathrm{gr}_{F}^{2}H_{1}(a)\oplus \mathrm{gr}_{F}^{1}H_{1}(a)\,,\\
\mathrm{gr}^{W}_{2}H(a)&=&W_{2}H(a)=
H_{e^{-{4\pi i \over 3}}}(a)
\oplus
H_{e^{-{2\pi i \over 3}}}(a)=\mathrm{gr}_{F}^{1}H_{e^{-{4\pi i \over 3}}}(a)\oplus
\mathrm{gr}_{F}^{1}H_{e^{-{2\pi i \over 3}}}(a)\,.
\end{eqnarray*}
The pure Hodge structure on $\mathrm{gr}^{W}_{3}H(a)$ is mapped to that on $H^{1}(Q_a)$ under the Poincar\'e residue map \eqref{eqnLGCY}.

\paragraph{$D$-modules for geometric sections}

Differential equations satisfied by various sectors are given in Lemma \ref{lemPicardFuchs}.
The results are displayed in Table \ref{table-differentialequationcubic}.
In particular, the differential operator with $\bold{m}=(1,1,1)$ follows from the one
with $\bold{m}=(0,0,0)$ according to \eqref{eqncommutingdifferentialoperators}.
\begin{table}[h]
  \centering
\caption{Differential equations satisfied by the geometric sections $s[\bold{z}^{\bold{m}}\Omega]$, and orbifold structures on $\mathcal{M}$ provided by the corresponding D-modules under the Riemann-Hilbert correspondence.}
  \label{table-differentialequationcubic}
  \renewcommand{\arraystretch}{1.5} 
 \begin{tabular}{c|c|c}
 \hline
$\bold{m}$ & differential operator& orbifold structure\\
 \hline
$(0,0,0)$ & $(\theta_{a}-1)\theta_{a}-{a^3\over 3^{3}} (\theta_a+1)^2$ & $T_{3,\infty,\infty}$\\
$(1,0,0)$ & $\theta_{a}-{a^3\over 3^{3}} (\theta_a+1)$ & $T_{1,3,3}$  \\
$(1,1,0)$ & $\theta_{a}-{a^3\over 3^{3}} (\theta_a+2)$ & $T_{1,3,3}$\\
$(1,1,1)$ & $(\theta_{a}-2)\theta_{a}-{a^3\over 3^{3}} (\theta_a+2)^2$ & $T_{3,\infty,\infty}$\\
\hline
\end{tabular}
\end{table}

\paragraph*{The $\bold{m}=(0,0,0)$ sector.}
The differential operator is
\begin{equation}\label{eqncubicmarginalbeforetwist}
\mathcal{D}_{\mathrm{cubic}}:=\theta_{b}(\theta_{b}-{1\over 3})
-b (\theta_{b}+{1\over 3}) (\theta_{b}+{1\over 3}) \,,\quad b={a^{3}\over 3^{3}}\,.
\end{equation}
By dimension reasons we have\footnote{Hereafter we shall often identify the flat bundle $D_{\bold{m}}$ on $\mathcal{M}$
with the corresponding local system and $H$ with $\mathrm{gr}_{\bullet}^{W}H$, without explicit mentioning.}
$$D_{\bold{0}}\cong C_{0}=H_{1}\cong \mathrm{gr}_{W}^3 H\,.$$
The period map is then
\begin{equation}
\mathcal{M}=\mathbb{C}-\{0,1\}\rightarrow  \mathrm{Gr}(1,\mathrm{gr}^{W}_{3}H(b_{*}))
\,,\quad
b\mapsto
F^{3}\mathrm{gr}^{W}_{3}H(b)\,,
\end{equation}
where $b_{*}$ is a reference point on $\mathcal{M}$.
The monodromy group can be computed following \cite{Erdelyi:1981} or \cite{Beukers:1989}.
To be explicit, a basis of solution is given by
$$u_{3}=\,_{2}F_{1}({1\over 3},{1\over 3};{3\over 3};b)\,,\quad
u_{4}=b^{1\over 3}\,_{2}F_{1}({2\over 3},{2\over 3};{4\over 3};b)\,.
$$
Then up to conjugacy, the monodromies around $b=0,1,\infty$ are given by the following matrices
in $\mathrm{SL}_{2}(\mathbb{R})\times U_{1}(\mathbb{R})$
\begin{equation}\label{eqnmonodromycubicmarginal}
T_{0}=\xi_{3}^{-1}\otimes
\begin{pmatrix}
-2 & -1\\
3 &1
\end{pmatrix}\,,
\quad
T_{1}=
\begin{pmatrix}
1 & 0\\
-3 &1
\end{pmatrix}\,,
\quad
T_{\infty}=\xi_{3}\otimes
\begin{pmatrix}
1 & 1\\
0 &1
\end{pmatrix}\,,\quad T_{0}T_{1}T_{\infty}=\mathbb{1}\,.
\end{equation}

The Schwarzian uniformization using solutions to
\eqref{eqncubicmarginalbeforetwist}
then provides the base $\mathcal{M}$
with the orbifold structure $T_{3,\infty,\infty}=\Gamma_{3,\infty,\infty}\backslash\mathcal{H}$.
Furthermore, the sub-bundle
$F^{1}D_{\bold{0}}$
is the automorphic bundle with a $j$-automorphy factor \eqref{eqnjautomorphy}
\begin{equation}\label{eqnjautomorphycubic}
j: \,\Gamma_{3,\infty,\infty}\times \mathcal{H}\rightarrow \mathrm{GL}_{1}(\mathbb{C})\,,\quad
( \begin{pmatrix}
a & b\\
c &d
\end{pmatrix},\tau)\mapsto (c\tau+d)^{-1}\,.
\end{equation}
The bundle $D_{\bold{0}}=C_{0}=\mathrm{gr}^{W}_{3}H$
 can be extended to a locally free sheaf $(\mathrm{gr}^{W}_{3}\overline{H},\overline{\nabla})$ equipped with meromorphic connection on the compactification $\overline{T}_{3,\infty,\infty}$
of $\mathcal{M}$.
Deligne's canonical extension gives rise to a connection with only logarithmic pole
\begin{equation*}
\overline{\nabla}: \mathrm{gr}^{W}_{3}\overline{H}\rightarrow \mathrm{gr}^{W}_{3}\overline{H}\otimes \Omega^{1}_{\overline{T}_{3,\infty,\infty}}(\log(p_1+p_{\infty}))\,,
\end{equation*}
where $p_0, p_{1},p_{\infty}$ stands for the divisor represented by $b=0,1,\infty$ on the coarse moduli of $\overline{T}_{3,\infty,\infty}$, respectively.
To be explicit, up to conjugation the corresponding exponent matrices are given by
\begin{equation*}\label{eqnexponentcubicmarginal}
L_0=\begin{pmatrix}
0 & 0\\
0 & {1\over 3}
\end{pmatrix}\,,\quad
L_1=\begin{pmatrix}
0 & 1\\
0 & 0
\end{pmatrix}\,,\quad
L_\infty=\begin{pmatrix}
{1\over 3} & 1\\
0 & {1\over 3}
\end{pmatrix}\,.
\end{equation*}
Concretely, this canonical extension is obtained by choosing suitable local solutions as the
trivialization.
See
\cite{Saber2014:vectorvalued, Candelori:2017vector, Candelori:2019vector} for details.

 \paragraph*{The $\bold{m}=(1,0,0)$ sector.}
 By the $\mathfrak{S}_{3}$-symmetry in Lemma \ref{lemPicardFuchs}, among the $3$ sectors in $H_{\lambda,\beta}, \lambda=\exp(-{4\pi i \over 3}),\beta={1/3}$,
 we only need to consider  $\bold{m}=(1,0,0)$.
For this sector (and any of its permutations), the differential operator is
$$\theta_{b}
-b (\theta_{b}+{1\over 3}) \,.$$
The solution space is spanned by the multi-valued function $(1-b)^{-{1\over 3}}$ on $\mathcal{M}$.
Under the Riemann-Hilbert correspondence, the differential equation gives a local system
\begin{equation*}
D_{\bold{m}}\cong C_{-{2\over 3}}\,.
\end{equation*}
This is an orbifold line bundle on
the orbifold $T_{1,3,3}$ with signature
$1,3,3$ at $b=0,1,\infty$ respectively.
The canonical extension
defines the following orbifold line bundle on
\begin{equation*}
   \overline{D}_{ \bold{m}}
\cong \mathcal{O}_{T_{1,3,3}}(p_{1})\otimes\mathcal{O}_{T_{1,3,3}}(p_{\infty})^{\otimes -1} \,.
\end{equation*}

 \paragraph*{ The $\bold{m}=(1,1,0)$ sector.}
 For $\bold{m}=(1,1,0)$ and permutations, the differential operator is
$$\theta_{b}
-b (\theta_{b}+{2\over 3}) \,.$$
The solution space is spanned by
$(1-b)^{-{2\over 3}}\,.$
The same reasonings as in the case $\bold{m}=(1,0,0)$ apply to this case.

\paragraph{Schwarzian uniformization and elliptic modular forms}

For the $D_{\bold{0}}$ sector,
applying a twist by $b^{-{1\over 3}}$ to
\eqref{eqncubicmarginalbeforetwist} leads to
the following differential operator
\begin{equation}\label{eqncubicmarginalaftertwist}
\mathcal{D}_{\mathrm{elliptic}}:=
b^{1\over 3}
\circ \mathcal{D}_{\mathrm{cubic}}
\circ b^{-{1\over 3}}=
(\theta_{b}-{2\over 3})(\theta_{b}-{1\over 3})
-b \theta_{b}^2\,,
\end{equation}
with unipotent instead of quasi-unipotent monodromies at $b=1,\infty$.
This differential operator coincides with the familiar Picard-Fuchs operator for the Hesse pencil as it should be the case by the
isomorphism of pure Hodge structures provided by the Poincar\'e residue map \eqref{eqnLGCY}.
Unlike the one in \eqref{eqnmonodromycubicmarginal},
the monodromy group is now a subgroup of $\mathrm{SL}_{2}(\mathbb{Z})$.
The
canonical extension is determined by the exponent matrices
\begin{equation*}
L_0=\begin{pmatrix}
{1\over 3} & 0\\
0 & {2\over 3}
\end{pmatrix}\,,\quad
L_1=\begin{pmatrix}
0 & 0\\
-3 & 0
\end{pmatrix}\,,\quad
L_\infty=\begin{pmatrix}
0& 1\\
0 & 0
\end{pmatrix}\,.
\end{equation*}

Using the connection between solutions to \eqref{eqncubicmarginalaftertwist} and modular forms \cite{Maier:2009}, one then
identifies the base $\mathcal{M}$
with  the modular curve
\begin{equation*}
 \mathcal{M}=\Gamma_{0}(3)\backslash \mathcal{H}^{*}-\{0,1,\infty\}
\subseteq T_{3,\infty,\infty}\cong \Gamma_{0}(3)\backslash \mathcal{H}= \Gamma_{0}(3)\backslash \mathcal{H}^{*}-\{1,\infty\}\,,
\quad \Gamma_{0}(3)\cong \Gamma_{3,\infty,\infty}\,.
\end{equation*}
Note that the twist does not affect the projectivized monodromy group.
Under the above identification, the Schwarzian for \eqref{eqncubicmarginalbeforetwist}
gets identified with the modular parameter $\tau$ for the elliptic curve family, and the
Hauptmodul $b$
is a modular function for the modular group $\Gamma_{0}(3)$.

First order differential equations for the other sectors $D_{\bold{m}},\beta(\bold{m})<1$,
which have semi-simple monodromies, also  correspond to representations of the modular group $ \Gamma_{0}(3)\cong \Gamma_{3,\infty,\infty}$.
 This is also the case for differential equation satisfied by the twist $b^{-{1\over 3}}$ itself, which was used to pass from \eqref{eqncubicmarginalbeforetwist}
to
\eqref{eqncubicmarginalaftertwist} for the  $\bold{m}=0$ sector.

Therefore,
when expressed in terms of the Schwarzian for the sector $D_{\bold{0}}$,
all of the twisted sectors are components of (vector-valued) elliptic modular  forms for the group $\Gamma_{0}(3)$ which arises from the Schwarzian
uniformization for the sector $D_{\bold{0}}$.
Similar statement also holds for sections of the extended bundle $F^{0}\overline{D}_{\bold{0}}, F^{1}\overline{D}_{\bold{0}}$.
In fact, the automorphic form $a=3b^{1\over 3}$ is an elliptic modular function for the principal congruence subgroup $\Gamma(3)$, while
 those of $F^{1}\overline{D}_{\bold{0}}\cong \mathcal{O}_{\overline{T}_{3,\infty,\infty}}(p_{0}) $ are weight-one modular forms
(with multiplier system given by the Dirichlet character $\chi_{-3}$).
 See \cite{Maier:2009, Maier:2011} for details.

\begin{rem}\label{remquasimodularforms}

One can alternatively regard sections of $F^{0}D_{\bold{0}}$ as quasi-modular forms.
To be more precise, one takes a frame for $F^{1}D_{\bold{0}}$ from which
one sees that the $j$-automorphy factor is $(c\tau+d)^{-1}$, as has been explained
 in \eqref{eqnjautomorphycubic} by using the Schwarzian uniformization.
Then one applies the covariant derivative $\nabla_{\partial_b}$
to obtain a frame of $F^{0}D_{\bold{0}}$.
The automorphy factor can then be read off from the resulting frame.
Coordinate expressions of sections for $F^{0}D_{\bold{0}}$ with respect to this frame are then
analogues of quasi-modular forms \cite{Kaneko:1995}, see e.g., \cite{Urban:2014nearly, Ruan:2020genus} for related discussions.
The same discussions apply to any rank-$2$ local system
arising from ordinary differential equations with regular singularities.
\end{rem}

\begin{rem}
While the geometric sections $s[\bold{z}^{\bold{m}}\Omega]$ are components of automorphic forms, interestingly
the coordinate functions $z_{0},z_{1},z_{2}$ themselves can be related to automorphic functions.
In fact, the Hesse pencil
$$z_{0}^{3}+z_{1}^{3}+z_{2}^{3}-a z_0 z_1 z_2=0\,$$
can be uniformized \cite{Dolgachev:modularform} by Jacobi theta functions for the modular group $\Gamma(3)$.
From the moduli point of view, the Hesse pencil is the universal family of elliptic curves equipped with level-$3$ structure (a frame for the group of $3$-torsion points) over the
modular curve $\Gamma(3)\backslash \mathcal{H}$.
The objects $z_i,i=0,1,2$ are sections of the pushforward of certain theta line bundles of degree $3$ over the universal curve,
while the Poincar\'e residue of
$s[\Omega]$ is a section of the Hodge line bundle attached to this elliptic curve family.
\end{rem}

\paragraph{Real structure}
The real structure on $H$ in the present case exchanges
the two types of summands in $\mathrm{gr}^{W}_{2}H$ with distinct values of $\beta$, and also
those in $\mathrm{gr}^{W}_{3}H$.
As explained earlier, the real (in fact integral) structure can be identified
by making use of the oscillating integrals.
The discussions for the $\mathrm{gr}^{W}_{3}H$ piece
is identical to those for the Hesse pencil and is therefore omitted here.
For the
 $\mathrm{gr}^{W}_{2}H$ piece, similar to
 the computations in
 \eqref{eqnoscilatingintegrals},
 \eqref{eqnoscilatingintegralsintermsofsumoverdelta},
  \eqref{eqnoscilatingintegralintermsofhypergeometric} in
Section
\ref{secintegralstructure}, we have for $\bold{m}=(1,0,0), (0,1,1)$ and thus their $\mathfrak{S}_{3}$-permutations
\begin{eqnarray*}
\int_{\gamma_{\bold{h}}}e^{-f_{a}}\bold{z}^{\bold{m}}\Omega
=\prod_{i=0}^{2}(1-\xi_{3}^{h_{i}(m_{i}+1)})\cdot (2\pi )^{2\over 2}3^{-{1\over 2}-3} {\prod_{i=0}^{2} \Gamma({m_{i}+1\over 3})
 \over \prod_{i=0}^{2} \Gamma({i+1\over 3})} \cdot
 \,_{3}G_{3}({m_0+1\over 3}, {m_1+1\over 3}, {m_2+1\over 3};{1\over 3},{2\over 3},{3\over 3}; b)
 \,.
\end{eqnarray*}
Denote the linear dual of $\gamma_{\bold{h}}$ by $\check{\gamma}_{\bold{h}}$, we then have
$${1\over c_{\bold{m}}}\cdot [\bold{z}^{\bold{m}}\Omega]:= { (2\pi )^{-1}3^{{1\over 2}+3} { \prod_{i=0}^{2} \Gamma({i+1\over 3})\over \prod_{i=0}^{2} \Gamma({m_{i}+1\over 3})}
\over _{3}G_{3}({m_0+1\over 3}, {m_1+1\over 3}, {m_2+1\over 3};{1\over 3},{2\over 3},{3\over 3}; b) }\cdot [\bold{z}^{\bold{m}}\Omega]
=
\sum_{\bold{h}\in \{0,1,2\}^{\oplus 3}}\prod_{i=0}^{2}(1-\xi_{3}^{h_{i}(m_{i}+1)}) \check{\gamma}_{\bold{h}} \,.$$
Introduce the Dirichlet character $\chi_{-3}$ that takes
the value $0,1,-1$ on $0,1,2$ modulo $3$.
Then it is easy to check that
$$(1-\xi_{3}^{h_{i}(m_{i}+1)})=\xi_{12}^{-\chi_{-3}(h_{i}(m_{i}+1))}$$
and hence
$$\prod_{i=0}^{2}(1-\xi_{3}^{h_{i}(m_{i}+1)})=\xi_{12}^{-\chi_{-3}(\bold{h})\cdot \chi_{-3}(\bold{m+1}) }\,,$$
where $\chi_{-3}(\bold{m+1})$ stands for the vector with components
$\chi_{-3}(m_0+1),\chi_{-3}(m_1+1),\chi_{-3}(m_2+1)$
and the notation $\cdot$ above stands for inner product between two vectors.
The property of the Dirichlet character $\chi_{-3}$ tells that
$$ \chi_{-3}(\bold{m+1})=- \chi_{-3}(\bold{1-m+1})\,,\quad \forall\, \bold{m}\in \{0,1\}^{\oplus 3}\,.$$
This immediately tells that
$$c_{\bold{1-m}}^{-1}\cdot [\bold{z}^{\bold{1-m}}\Omega]=\overline{ c_{\bold{m}}^{-1}\cdot [\bold{z}^{\bold{m}}\Omega]}\in H\otimes \mathbb{C}\,.$$
In particular, the real cohomomology classes are spanned over $\mathbb{R}$
by
$$\mu c_{\bold{m}}^{-1}\cdot [\bold{z}^{\bold{m}}\Omega]+\bar{\mu} c_{\bold{1-m}}^{-1}\cdot [\bold{z}^{\bold{1-m}}\Omega]
=\sum_{\bold{h}\in \{1,2\}^{\oplus 3}}( \mu\xi_{12}^{-\chi_{-3}(\bold{h})\cdot \chi_{-3}(\bold{m+1})}+\bar{\mu}
\xi_{12}^{\chi_{-3}(\bold{h})\cdot \chi_{-3}(\bold{m+1})}
)\check{\gamma}_{\bold{h}}\,,$$
where $\mu\in \mathbb{C},\bold{m}\in \{0,1\}^{\oplus 3}$.
Integral classes can then be obtained by suitably choosing $\mu$.

\subsubsection{Fermat quartic}

For the $n=3$ case, the spectrum
is shown in Table \ref{table-spectrumquartic} below.
\begin{table}[h]
  \centering
\caption{Spectrum of the quartic  polynomial singularity $f_{a}$\,.}
  \label{table-spectrumquartic}
  \renewcommand{\arraystretch}{1.5} 
 \begin{tabular}{c|ccccccccc}
 \hline
$\beta$ & $0$ & ${1\over 4}$ & ${2\over 4}$ & ${3\over 4}$ & $1$ & ${5\over 4}$ & ${6\over 4}$ & ${7\over 4}$ & $2$\\
 \hline
$n_{\beta}$  & $1$ & $4$ & $10$ & $16$ & $19$ & $16$ & $10$ & $4$ & $1$\\
\hline
\end{tabular}
\end{table}

\paragraph{Automorphic bundles}

We will be mainly interested in the $0\leq \beta(\bold{m})< 1$ sectors: the discussions on the other sectors are similar according to \eqref{eqncommutingdifferentialoperators}.
The equations in Lemma \ref{lemPicardFuchs} are listed in Table \ref{table-differentialequationquartic}.
\begin{table}[h]
  \centering
\caption{Differential equations satisfied by the geometric sections $s[\bold{z}^{\bold{m}}\Omega]$.}
  \label{table-differentialequationquartic}
  \renewcommand{\arraystretch}{1.5} 
 \begin{tabular}{c|c}
 \hline
$\bold{m}$ & differential operator\\
 \hline
$(0,0,0,0)$ & $(\theta_{a}-2)(\theta_{a}-1)\theta_{a}-{a^4\over 4^{4}} (\theta_a+1)^3$\\
$(1,0,0,0)$ & $(\theta_{a}-1)\theta_{a}-{a^4\over 4^{4}} (\theta_a+1)^2$\\
$(2,0,0,0)$ & $(\theta_{a}-2)\theta_{a}-{a^4\over 4^{4}} (\theta_a+1)^2$\\
$(1,1,0,0)$ & $(\theta_{a}-1)\theta_{a}-{a^4\over 4^{4}} (\theta_a+1)(\theta_a+2) $\\
$(1,1,1,0)$ & $(\theta_{a}-1)\theta_{a}-{a^4\over 4^{4}} (\theta_a+2)^2 $\\
$(2,1,0,0)$ & $\theta_{a}-{a^4\over 4^{4}}(\theta_a+1) $\\
\hline
\end{tabular}
\end{table}

\paragraph*{The $D_{\bold{0}}$ sector.}
We first consider the sector with
$\bold{m}=(0,0,0,0)$ given by
\begin{equation}\label{eqnPFK3beforetwist}
\mathcal{D}_{\mathrm{quartic}}:=\theta_{b}(\theta_{b}-{1\over 4})(\theta_{b}-{2\over 4})-b(\theta_{b}+{1\over 4})^{3}\,,
\end{equation}
with the corresponding local system $D_{\bold{0}}$ over $\mathcal{M}$.
Similar to the cubic case, one can show that sections of  $D_{\bold{0}}$
correspond to automorphic forms for the triangular group
$\Gamma_{4,2,\infty}\cong \Gamma_{0}(2)^{+}<\mathrm{PSL}_{2}(\mathbb{R})$, up to a twist by $b^{-{1\over 4}}$
that changes $\mathcal{D}_{\mathrm{quartic}}$ to $\mathcal{D}_{\mathrm{K3}}:
=b^{1\over 4}\circ \mathcal{D}_{\mathrm{quartic}}\circ b^{-{1\over 4}}$.
The twist $b^{1\over 4}$ itself can be regarded as
a section of a flat bundle and hence an automorphic form for $\Gamma_{4,2,\infty}$,
by the Riemann-Hilbert correspondence.
Furthermore one can show that after a further twist, these sections are related to elliptic modular forms for $\Gamma_{0}(2)$.
In order not to interrupt the main stream of the presentation, we have left the detailed discussions to
Appendix \ref{secmarginalquartic}.

\paragraph*{Sectors with $\beta(\bold{m})=1/4$.}
Again by the $\mathfrak{S}_{4}$-symmetry in  Lemma \ref{lemPicardFuchs}, we only need to focus on  $\bold{m}=(1,0,0,0)$.
Results for the other sectors obtained by $\mathfrak{S}_{4}$-actions are the same.
From the results in Table \ref{table-differentialequationquartic}, we see that
for this sector the differential operator is
$$\theta(\theta-{1\over 4})-b(\theta+{1\over 4})^2\,,\quad \theta:=b{\partial\over \partial b}\,.$$
The corresponding bundle
$D_{\bold{m}}$
is a representation of the triangular group
$\Gamma_{4,4,\infty}$  that arises as the projectivized monodromy group.
Furthermore, sections of
$F^{2}D_{\bold{m}}$ and $F^{1}D_{\bold{m}}$ are automorphic forms for this triangular group.
In fact, similar to any other rank-$2$ local system arising from a Fuchsian equation,
from the Schwarzian uniformization one can directly read off the automorphy factors for
these bundles  as explained in Remark
\ref{remquasimodularforms}.

\paragraph*{Other twisted sectors with $\beta(\bold{m})<1$.}
The same reasoning  gives the triangular groups (essentially determined by the exponent differences)
for all the other twisted sectors with $\beta(\bold{m})<1$. Results from straightforward computations
 complete
 Table \ref{table-differentialequationquartic}
to
 Table   \ref{table-monodromiesgroupquartic}
below.

\begin{rem}

To demonstrate the analogy between automorphic forms in $D_{\bold{m}}$ and quasi-modular forms, consider the $\bold{m}=\bold{0}$ case.
According to the computations in Appendix \ref{secmarginalquartic}, the automorphy factor for
$F^{3}D_{\bold{0}}$ is $(c\tau+d)^{-2}$, with $\tau$ being the Schwarzian given in \eqref{eqnSchwarzianquarticCY}.
Then the automorphy factors for $F^{2}D_{\bold{0}}$ and $F^{1}D_{\bold{0}}=F^{0}D_{\bold{0}}$ are given by
\begin{equation*}
\begin{pmatrix}
(c\tau+d)^{-2} &  -2c(c\tau+d)^{-1} &\\
0 & 1 & &\\
\end{pmatrix}\,,\quad
\begin{pmatrix}
(c\tau+d)^{-2} &  -2c(c\tau+d)^{-1} & 2c^2(c\tau+d)^{-2}\\
0 & 1 & -2c(c\tau+d) &\\
0 & 0 & (c\tau+d)^{2}
\end{pmatrix}\,,
\end{equation*}
respectively. Coordinate expressions for the corresponding automorphic form
then satisfy transformation laws similar to those of quasi-modular forms.
\end{rem}

\begin{rem}
The triangular groups $\Gamma_{\bold{\ell}(\bold{m})}$ are regarded as subgroups of $\mathrm{PSL}_{2}(\mathbb{R})$ arising from the projectivized monodromy groups.
However, the presentations of  the monodromies (hypergeometric group in the terminology of  \cite{Beukers:1989} using Levelt's theorem) a priori are only matrices in $\mathrm{GL}_{2}(\mathbb{C})$.
Basing on the results in \cite{Beukers:1989} on the monodromy groups of hypergeometric equations,
up to conjugacy
they are actually elements in $\mathrm{SL}_{2}(\mathbb{R})\times U_{1}(\mathbb{R})$.
These monodromy matrices are explicitly worked out and also included in Table  \ref{table-monodromiesgroupquartic}.
\begin{table}[h]
  \centering
\caption{Triangular groups and monodromy representations corresponding to cohomology sectors.
Here $\xi_{k}=\exp(2\pi i/k)$, and $T_{1}$ is determined from the relation $T_{\infty}T_{1}T_{0}=\mathbb{1}$.
The symmetric square of the monodromy representation of the row labelled by $\bold{m}=-$  is isomorphic to the one for the $\bold{m}=\bold{0}$ case (see Appendix \ref{secmarginalquartic}).
The last row with $\bold{m}=\star$ shows an auxiliary differential equation that does nor arise from any sector but has $\Gamma_{4,\infty,\infty}$ as its projectivized monodromy group.}
  \label{table-monodromiesgroupquartic}
  \renewcommand{\arraystretch}{1.2} 
 \begin{tabular}{c|c|c|c}
 \hline
$\bold{m}$ & differential equation &  group& presentations of $ T_{0},T_{\infty}$\\
 \hline
$(0,0,0,0)$
 & $\theta(\theta-{1\over 4})(\theta-{2\over 4})-b(\theta+{1\over 4})^{3}$
& $\Gamma_{4,2,\infty}$ &
$
\xi_{4}^{-1}
\begin{pmatrix}
0 &0 & -1\\
0 & 1 &4\\
-1 & 1 & -2
\end{pmatrix}\,
\,,
\xi_{4}
\begin{pmatrix}
1 & -1 & 2\\
0 & 1 & -4\\
0& 0 &1
\end{pmatrix}
$
\\
$(1,0,0,0)$
& $\theta(\theta-{1\over 4})-b(\theta+{1\over 4})^2$
&  $\Gamma_{4,4,\infty}$
&
$
\xi_{8}
{1\over \sqrt{2}}
\begin{pmatrix}
-1 & -5\\
1 &3
\end{pmatrix}
\,,
\xi_{4}
\begin{pmatrix}
1 & -4\\
0 &1
\end{pmatrix}
$
\\
$(2,0,0,0)$
& $\theta(\theta-{2\over 4})-b(\theta+{1\over 4})^2$
&$\Gamma_{2,\infty,\infty}$
&
$
\xi_{4}^{-1}
\begin{pmatrix}
0 & -1\\
1 &0
\end{pmatrix}
\,,
\xi_{4}
\begin{pmatrix}
0 & -1\\
1 &2
\end{pmatrix}
$
\\
$(1,1,0,0)$
& $\theta(\theta-{1\over 4})-b(\theta+{1\over 4})(\theta+{2\over 4})$
& $\Gamma_{4,\infty,4}$&
$
\xi_{8}^{-3}
\begin{pmatrix}
-\sqrt{2} & -1\\
1 &0
\end{pmatrix}
\,,
\xi_{8}^3
\begin{pmatrix}
0 & 1\\
-1 &\sqrt{2}
\end{pmatrix}
$
\\
$(1,1,1,0)$
&  $\theta(\theta-{1\over 4})-b(\theta+{2\over 4})^2$
& $\Gamma_{4,4,\infty}$&
$
\xi_{8}
{1\over \sqrt{2}}
\begin{pmatrix}
-1 & 1\\
-5 &3
\end{pmatrix}
\,,
\begin{pmatrix}
1 & -2\\
2 &-3
\end{pmatrix}
$
\\
$(2,1,0,0)$
& $\theta-b(\theta+{1\over 4}) $
& $\Gamma_{1,4,4}$ & $1\,,\quad \xi_{4} $\\
\hline
$-$
&  $\theta(\theta-{1\over 4})-b(\theta+{1\over 8})^2$
& $\Gamma_{4,2,\infty}$&
$
\xi_{8}^{-1}
\begin{pmatrix}
{\sqrt{2}-1\over 2} & -{3\over 2}\\
{1\over 2} & {\sqrt{2}+1\over 2}
\end{pmatrix}
\,,
\xi_{8}
\begin{pmatrix}
{3\over 2} & -{\sqrt{2}+1\over 2}\\
 {\sqrt{2}-1\over 2} & {1\over 2}
\end{pmatrix}
$
\\
$\star$
&  $\theta(\theta-{3\over 4})-b(\theta+{1\over 8})^2$
& $\Gamma_{4,\infty,\infty}$&
$
\xi_{8}^{-1}
\begin{pmatrix}
\sqrt{2} & -1\\
1 &0
\end{pmatrix}
\,,
\xi_{8}
\begin{pmatrix}
0 & 1\\
-1 &2
\end{pmatrix}
$
\\
\hline
\end{tabular}
\end{table}
The $U_{1}(\mathbb{R})$ factors can be cancelled by a suitable twist which can be multi-valued functions on the quotient $\Gamma_{\bold{\ell}(\bold{m})}\backslash\mathcal{H}$.
The twist procedure also kills the semi-simple part of the quasi-unipotent monodromies.
For example, for the $\bold{m}=(1,0,0,0)$ case, a twist by $b^{-{1\over 8}}(1-b)^{3\over 8}$ changes the differential operator $\theta(\theta-{1\over 4})-b(\theta+{1\over 4})^2$ to
$$b^{-{1\over 8}}(1-b)^{3\over 8}\circ \left(\theta(\theta-{1\over 4})-b(\theta+{1\over 4})^2\right)\circ (b^{-{1\over 8}}(1-b)^{3\over 8})^{-1}\,.$$
The latter is a differential operator whose Riemann scheme  is given by Table   \ref{table-orbifoldstructurecubicweight2degree1/4}.
\begin{table}[h]
  \centering
\caption{Riemann scheme for the differential equation
for the $\bold{m}=(1,0,0,0)$ sector after the twist.}
  \label{table-orbifoldstructurecubicweight2degree1/4}
  \renewcommand{\arraystretch}{1.5} 
 \begin{tabular}{c|ccc}
 \hline
$b$ & $0$ & $1$ & $\infty$ \\
 \hline
local exponents   & $-{1\over 8}$ & ${3\over 8}$ & $0$ \\
   & ${1\over 8}$ & ${5\over 8}$ & $0$ \\
\hline
\end{tabular}
\end{table}
For the $\bold{m}=(2,0,0,0)$ case, applying the twist $b^{-{1\over 4}}$, one arrives at the familiar presentation of
 $\Gamma_{2,\infty,\infty}\cong \Gamma_{0}(2)<\mathrm{PSL}_{2}(\mathbb{Z})$ in terms of matrices in $\mathrm{SL}_{2}(\mathbb{Z})$.

\end{rem}

\paragraph{Unifying triangular groups}

With the triangular groups $\Gamma_{\bold{\ell}(\bold{m})}$
 determined from the sectors $D_{\bold{m}}$,  the desired quotient homomorphism
$\rho_{\bold{\ell}(\bold{m})}: \Gamma_{4, \infty, \infty}\rightarrow
\Gamma_{\bold{\ell}(\bold{m})} $
in the proof Theorem \ref{thmautomorphicform}
is in no way unique.
We can however construct, once and for all, natural ones by using the fact that they are
orbifold fundamental groups.
Such homomorphisms can be chosen to respect the semi-simple parts of their monodromy representations.
Concretely, to obtain a homomorphism
one maps the generators of
$\Gamma_{4,\infty,\infty}$ to suitable powers of those of $ \Gamma_{\ell_0, \ell_1, \ell_\infty}$.
The powers are dictated by the semi-simple parts of the monodromies.
For example,
using the presentation in \eqref{eqnpresentationtriangulargroup}, a homomorphism
$\Gamma_{4, \infty, \infty}\rightarrow\Gamma_{2, \infty, \infty}$
can be given by $\sigma\mapsto \sigma',\tau\mapsto \tau'$, where
$\sigma,\tau$ are the generators corresponding to the loops at the orbifold point $p_0$ and the cusp
$p_{\infty}$, and similarly
$\sigma', \tau'$ the corresponding generators for $\Gamma_{2, \infty, \infty}$.
Note that this homomorphism  does not descend from a homomorphism on $\mathrm{PSL}_{2}(\mathbb{R})$
as the monodromy matrices corresponding to $\sigma,\tau$
do not generate the full group $ \mathrm{PSL}_{2}(\mathbb{R})$ but only a subgroup.

\subsubsection{Fermat quintic}
\label{secquinticexample}

The discussions for the $n=4$ case are analogous to the previous cases, except for
the lack of connection to elliptic modular forms and the explicit identification of automorphy
factors for $F^{\mathrm{top}}D_{\bold{m}}$.
Here we again only focus on the $D_{\bold{0}}$ sector which exhibit particularly interesting properties.

As explained before, via the Poincar\'e residue map \eqref{eqnLGCY} the $D_{\bold{0}}$ sector coincides with the variation of Hodge structure for the
quintic family $X\rightarrow \mathcal{M}=\mathbb{P}^{1}-\{0,1,\infty\}$ given by
\begin{equation}
\sum_{i=0}^{4} z_{i}^{5}-  5 z^{-{1\over 5}}  \prod_{i=0}^{4} z_{i}=0\,.
\end{equation}
The Picard-Fuchs operator is now, up to a twist,
\begin{equation}\label{eqnPFequationqunitic}
\mathcal{D}_{\mathrm{quintic}}=\theta^{4}-z\prod_{k=1}^{4} (\theta+{k\over 5})\,,
\quad \theta:={z}{\partial\over \partial z}\,.
\end{equation}
By examining the indicial differences, the flat vector bundle $D_{\bold{0}}$
corresponds to a representation
of $\Gamma_{\infty,\infty,5}=\pi_{1}^{\mathrm{orb}}(T_{\infty,\infty,5})$.
In particular, sections of the canonical extension $\overline{D}_{\bold{0}}$ are automorphic forms for the triangular group $\Gamma_{\infty,\infty,5}$.
\begin{rem}
A crucial difference from the previous cases is that now the automorphy factor for
$F^{3}D_{\bold{0}}$
is not given by a $j$-automorphy factor \eqref{eqnjautomorphy}, this is partially due to the lack of a direct relation between the representation
$D_{\bold{0}}$ and that for the rank-$2$ fundamental representation of $\Gamma_{\infty,\infty,5}$.
By calculating the monodromy representation of $D_{\bold{0}}$ one can see that
it does not admit a symmetric cube structure, unlike the Fermat quartic case discussed in Appendix \ref{secmarginalquartic}.
This is ultimately related to the nontriviality of genus zero GW invariants for its mirror, see
\cite{Zhou:2016mirror} for extended discussions on this.
\end{rem}

We now study the differential structure of the periods, namely solutions to the Picard-Fuchs equation \eqref{eqnPFequationqunitic}, as promised in Section \ref{secRHcorrespondence}.
For concreteness, we take a basis of solutions near the maximally unipotent monodromy point $z=0$ and denote them by $I_{0},I_{1},I_{2},I_{3}$, respectively.
Such a basis can be obtained from the Frobenius method for instance.
\begin{dfn}\label{dfnYamaguchiYauring}
Denote
\begin{equation}
T=I_{1,0}={I_{1}\over I_{0}}\,,\quad
 I_{2,0}={I_{2}\over I_{0}}\,,\quad
  I_{3,0}={I_{3}\over I_{0}} \,.
\end{equation}
The Yamaguchi-Yau ring
\cite{Yamaguchi:2004bt} (see also \cite{Bershadsky:1993cx}) is given by (here $'=\partial_z$)
\begin{equation}\label{eqnhollimitYYgenerators}
\mathcal{F}_{\mathrm{YY}}:=\mathbb{C}(z)[
{I_{0}'\over I_{0}},
{I_{0}''\over I_{0}},
{I_{0}'''\over I_{0}}
,
{T''\over T'}]\,.
\end{equation}
\end{dfn}
By using the special geometry relation \cite{Strominger:1990, Freed:1999} satisfied by the Picard-Fuchs equation \eqref{eqnPFequationqunitic} above,
we have the following relations
\cite{Ceresole:1992su, Bershadsky:1993cx, Ceresole:1993qq, Lian:1994zv, Yamaguchi:2004bt}
 (see also \cite{Almkvist:2004differential, Alim:2007qj, Grimm:2007tm, Hosono:2008ve,  Zagier:2008})
\begin{eqnarray}\label{eqnrelationsforYY}
I_{3,0}'&=&-2I_{2,0}+I_{1,0} I_{2,0}\,,\nonumber\\
I_{2,0}''&= &I_{2,0}' {I_{1,0}''\over I_{1,0}'}+ {1\over I_{0}^2 I_{1,0}'}C\,,\quad C:={5\over 1-5^5 z}\,,\nonumber\\
T'''
&=&
T''(-2  {I_{0}'\over I_{0}}
+ C)
+ T' \left(   -4{ I_{0}''\over I_{0}}  +2 ({I_{0}'\over I_{0}})^2  - C'+2 {I_{0}'\over I_{0}} C+C^2 +{7\over 5}C\right)
\,.
\end{eqnarray}
See Appendix B of the arXiv version of \cite{Chang:2020} and references therein for the geometric explanation for each of the relations above.
In particular, the first relation follows from from
the fact that the monodromy lies in $\mathrm{SL}_{4}(\mathbb{C})$, while the last one from the flatness of the Gauss-Manin connection.

\begin{prop}\label{propYYdifferentialring}
The Yamaguchi-Yau ring $\mathcal{F}_{\mathrm{YY}}$ is a differential ring
under $\partial_z$, with generators being components of automorphic forms for $\Gamma_{\infty,\infty,5}$.
\end{prop}
\begin{proof}
The proof of the first part is contained in \cite{Yamaguchi:2004bt}. It follows from the Picard-Fuchs equation satisfied by $I_0$ and
the last relation in \eqref{eqnrelationsforYY}.
The second part follows from Theorem \ref{thmautomorphicform}, the Poincar\'e residue map \eqref{eqnLGCY},
and the fact that period integrals of $Q_{a}$ are coordinate expressions of  cohomology classes in $H^{3}(Q_{a})$ with respect to the locally constant
frame of the corresponding local system over $\mathcal{M}$.
\end{proof}

The differential ring structure of  $\mathcal{F}_{\mathrm{YY}}$ has made its frequent appearance  in the study of higher genus mirror symmetry for the quintic and related geometries.
See \cite{Bershadsky:1993cx, Yamaguchi:2004bt, Alim:2007qj, Grimm:2007tm, Hosono:2008ve,  Zagier:2008} and the more recent works
\cite{Alim:2013eja, Lho:2018holomorphic, Chang:2018polynomial, Guo:2018structure}
for details.

We now prove the following result regarding the algebraic independence over $\mathbb{C}(z)$ of the generators in the Yamaguchi-Yau ring.
This property
 is of crucial importance in studying higher genus GW theory, such as solving holomorphic anomaly equations \cite{Alim:2013eja, Zhou:2014thesis, Lho:2018holomorphic}.
\begin{thm}\label{thmalgebraicindependence}
One has
\begin{equation}\label{eqnalgebraicindependenceYYgenerators}
\mathrm{trdeg}_{\mathbb{C}(z)} \,
\mathbb{C}(z)
\left({I_{0}'\over I_{0}},
{I_{0}''\over I_{0}},
{I_{0}'''\over I_{0}}
,
{T''\over T'}
\right)=4\,.
\end{equation}
In particular, the Yamaguchi-Yau generators are algebraically independent over $\mathbb{C}(z)$.
\end{thm}

\begin{proof}
The proof is a simple application of differential Galois theory.
Denote
$\mathcal{F}=\mathbb{C}(z)$ to be the ground differential field, with differential $\partial=\partial_{z}$.
Consider the Picard-Vessiot extension $\mathcal{G}$ for $\mathcal{D}_{\mathrm{quintic}}$ over $\mathcal{F}$.
Concretely this is the differential field over $\mathcal{F}$ obtained by adjoining differentials of the solutions to $\mathcal{D}_{\mathrm{quintic}}$.
Using the results between the monodromy group and the differential Galois group \cite{Beukers:1989},
a direct computation on the monodromy group shows that the transcendental degree satisfies
\begin{equation}
\mathrm{trdeg}_{\partial}(\mathcal{G}/\mathcal{F})=\dim \mathrm{Sp}_{4}(\mathbb{C})=10\,.
\end{equation}
Hence as ordinary fields
\begin{equation}\label{eqnupperbound}
\mathrm{trdeg} (\mathcal{G}/\mathbb{C}(z)) \geq 10\,.
\end{equation}
Since the periods satisfy the 4th order Picard-Fuchs equation \eqref{eqnPFequationqunitic}, we see that
a set of generators of $\mathcal{G}/\mathcal{F}$ can be taken to be the following $16$ elements
\begin{equation*}\label{eqngeneratorsYY}
I_{0}\,, I_{0}' \,, I_{0}''\,, I_{0}'''\,,T\,,T'\,,T''\,,T'''\,, I_{2,0}\,, I_{2,0}'\,, I_{2,0}''\,, I_{2,0}'''\,,
 I_{3,0}\,, I_{3,0}'\,, I_{3,0}''\,, I_{3,0}'''\,.
 \end{equation*}
 The relations
\eqref{eqnrelationsforYY}
tells that
\begin{equation}\label{eqnreducedPVextension}
\mathcal{G}=\mathcal{F}(I_{0}, I_{0}',I_{0}'',I_{0}''',T,T',T'', I_{2,0}, I_{2,0}',
 I_{3,0})=\mathcal{F} ( {I_{0}'\over I_{0}},{I_{0}''\over I_{0}} , {I_{0}'''\over I_{0}}, {T''\over T'},  I_{0}, T,T',  I_{2,0}, I_{2,0}',
 I_{3,0})
\,.
\end{equation}
In particular, one has
\begin{equation}\label{eqnlowerbound}
\mathrm{trdeg} (\mathcal{G}/\mathbb{C}(z)) \leq 10\,.
\end{equation}
Combining \eqref{eqnupperbound}, \eqref{eqnlowerbound}, we are led to
\begin{equation}\label{eqntranscendentaldegree}
\mathrm{trdeg} (\mathcal{G}/\mathbb{C}(z)) =10
\,.
\end{equation}
The desired assertion \eqref{eqnalgebraicindependenceYYgenerators}
now follows immediately.
\end{proof}

We remark that
a different way of finding the relations is presented in \cite{Movasati:2011}, which also leads to the
proof of \eqref{eqnlowerbound} and  in particular the algebraic independence
of the generators exhibited  in \eqref{eqnreducedPVextension}.

\section{Genus zero GW invariants of Fermat Calabi-Yau varieties}
\label{secmirrorsymmetry}

We are interested in the genus zero  Gromov-Witten invariants
corresponding to the part of the quantum cohomology generated by
classes pulled back from the ambient space. This ambient part is the part that enters genus zero mirror symmetry \cite{Candelas:1990rm}.
Specializing to our examples,
this is the identification  \cite{Givental:1996, Givental:1998, Lian:1997, Lian:1998} of  the $I$-functions, which encode
the genus zero GW invariants for the Calabi-Yau varieties,
with solutions to Picard-Fuchs equations for the mirror  \cite{Batyrev:1994hm} Calabi-Yau family given in  \eqref{eqnCYB}.
See \cite{Cox:1999} and also \cite{Gross:2003} for nice expositions.

We  now transfer results from singularity theory to results in genus zero GW theory.
\begin{thm}\label{thmgenuszeroGWautomorphicform}
Consider the genus zero Gromov-Witten invariants for the Calabi-Yau $(n-1)$-fold $M$
defined by the vanishing locus of the Fermat polynomial $\sum_{i=0}^{n}x_{i}^{n+1}=0$ in $\mathbb{P}^{n}$.
Then
the $I$-function, as a generating series of certain one-point GW invariants of $M$,
is a component of an automorphic form valued in $H^{*}(M)$ for the triangular group
$\Gamma_{n+1,\infty,\infty}$.
In particular, for the $n=2,3$ cases, the automorphic form is a vector-valued elliptic modular form (with possibly nontrivial multiplier systems) for certain congruence subgroup in
$\mathrm{PSL}_{2}(\mathbb{Z})$.
\end{thm}
\begin{proof}
By genus zero mirror symmetry, the $I$-function for $M$
is identical to the holomorphic volume form
for the corresponding Dwork family $\{Q_{a}\}_{a\in S}$ over $S$ in \eqref{eqnCYB}.
Invoke the fact that the Poincar\'e residue map in \eqref{eqnLGCY} respects the $D_{S}$-structures (i.e., Gauss-Manin connections),
to prove the holomorphic volume form is a component of an automorphic form it suffices to prove that for the section $\zeta[\Omega]$ in the $D_{\bold{0}}$ sector,
which in turn is equivalent to proving that for the
geometric section
$s[\Omega]$ by the quasi-homogeneity.
The last statement follows from Theorem \ref{thmautomorphicform}.

The assertions on elliptic modular forms for the $n=2,3$ cases follow from direct computations on the $D_{\bold{0}}$ sectors
shown in Section \ref{secexamples}.
\end{proof}

In an earlier work
\cite{Iritani:2021} it is shown, among other things, that all genus GW generating functions of certain quotients of Fermat CY varieties are ``abstract modular forms“ based on  Givental’s quantization formalism.	
In principle, one should be able to relate their modularity result
to the automorphicity result in our work.

We remark that our strategy in establishing automorphicity is applicable to more general invariants
(e.g., higher genus invariants that can be reconstructed from  \cite{Dubrovin:2001, Teleman:2012}
or reduced to \cite{Faber:2010} genus zero ones)
as long as the differential equations can be worked out.\\

In what follows we only discuss some quotients of quartic K3 surfaces.
The orbifold GW invariants and quantum cohomology of an orbifold are defined on the Chen-Ruan cohomology
\cite{Chen:2004, Adem:2007}
which carries sophisticated gradings.
Elements in the associated graded of the multi-graded Chen-Ruan cohomology are also called twisted sectors.
We shall expose
the roles played by twisted sectors in Chen-Ruan cohomology and twisted sectors in vanishing cohomology and study their mirror symmetry.

\subsection{$I$- and $J$-functions}

In the rest of this section, we denote by $Q$ the K3 surface defined by the vanishing locus of
$x^{4}+y^{4}+z^{4}+w^{4}=0$ in $\mathbb{P}^{3}$. We consider the
Givental $I$- and $J$-functions for
several quotients of $Q$.

\subsubsection{Quartic K3 surface}
The $I$-function $I_{Q}$ is a $H^{*}(Q)$-valued function satisfying
 the quantum differential equation
\begin{equation}\label{eqnquantumdifferentialequationquarticK3}
\left(\partial_{\mathfrak{t}}^{3}-4^{3} e^{\mathfrak{t}}\prod_{k=1}^{3}(\partial_\mathfrak{t}+{k\over 4})\right)I_{Q}(\mathfrak{t},\mathfrak{z}; P,\Lambda)=\mathfrak{z}({P\over \mathfrak{z}})^3e^{\mathfrak{t} P/\mathfrak{z}}=0\in   \mathfrak{z}\, H^{*}(Q)[[\mathfrak{z}^{-1},e^{\mathfrak{t}},\Lambda]]
\,,
\end{equation}
where $P$ is the hyperplane class of $\mathbb{P}^{3}$,
$\mathfrak{t}$ is the complexified K\"ahler parameter with $\mathrm{Re}\,\mathfrak{t}<0$,  $\mathfrak{z}$ is the descendant variable, and
$\Lambda$ is the Novikov variable which can be set to $1$ safely.
Explicitly it is given by
\begin{equation}\label{eqnIquarticK3}
I_{Q}(\mathfrak{t},\mathfrak{z}; P,\Lambda)=\mathfrak{z}e^{\mathfrak{t}P\over \mathfrak{z}} \sum_{d\in \mathbb{Z}_{\geq 0}}\Lambda^{4d}
e^{4d \mathfrak{t}}
{\prod_{\ell=1}^{4d}(4P+ \ell \mathfrak{z})
\over
\prod_{j=0}^{3}\prod_{\ell=1}^{d}(P+ \ell \mathfrak{z}) }=\mathfrak{z} F(\mathfrak{t})+  G(\mathfrak{t})+\mathcal{O}(\mathfrak{z}^{-1})
\,.
\end{equation}

The $\mathfrak{z}^{-1}$-expansion of $I_{Q}$ is the same as the one obtained by using the Frobenius method in solving differential equations. In particular, setting $\Lambda=1$ one has
\begin{equation}\label{eqnFquarticK3}
F(\mathfrak{t})=\sum_{d\geq 0} e^{d \mathfrak{t}}{4d\,!\over (d\,!)^{4}}
=
\,_{3}F_{2}({1\over 4},{2\over 4},{3\over 4}; 1,1; 4^{4} e^{\mathfrak{t}})\,.
\end{equation}
The small $J$-function $J_{Q}$, as a cohomology valued generating series of GW one-point functions,
is then given by
\begin{equation}\label{eqnJquarticK3}
J_{Q}(T,\mathfrak{z})={1\over F(\mathfrak{t})} I_{Q}(\mathfrak{t},\mathfrak{z})\,,\quad
 T={G(\mathfrak{t})\over F(\mathfrak{t})}\,.
 \end{equation}
 The big $J$-function can be obtained using axioms of GW theory.
 Expanding it in terms of $P$ tells that
the genus zero GW invariants are essentially trivial.
Computationally this can be seen from the symmetric structure of the above quantum differential equation, see Appendix \ref{secmarginalquartic}
for detailed discussions.

 \subsubsection{Minimal quotient}

 We consider the quotient of $Q$ by the trivial action
\begin{equation}
k\in G_{\mathrm{min}}= \mathbb{Z}/4\mathbb{Z}:\,
 [x,y,z,w]\mapsto [e^{2\pi i k \over 4}x, e^{2\pi i k \over 4}y, e^{2\pi i k
  \over 4}z,e^{2\pi i k \over 4}w ]\,.
\end{equation}
 The quotient is the stack
\begin{equation}
Y_{\mathrm{min}}=Q\times B\mathbold{\mu}_{4}\,,
\end{equation}
 where $\mathbold{\mu}_{4}\cong \mathbb{Z}/4\mathbb{Z}$
 is the multiplicative group consisting of 4th roots of unity and $B\mathbold{\mu}_{4}$ is the corresponding classifying space.

We follow the method in  \cite{Coates:2019}
to obtain the Givental $I$- and thus $J$-function by
using extended stacky fans.
First one takes the stacky fan $(N,\Sigma,\rho)$ that defines the
toric Deligne-Mumford stack $\mathbb{P}^{3}\times B\mathbold{\mu}_{4}
$
\begin{equation}
\beta:\mathbb{Z}^{4}\rightarrow N=\mathbb{Z}^{3}\oplus (\mathbb{Z}/4\mathbb{Z})\,,
\end{equation}
with
$\beta$ given by the matrix
$$
\left(
  \begin{array}{cccc}
-1 & 1 & 0 & 0\\
-1 & 0& 1 & 0\\
-1 & 0 & 0 & 1\\
\hline
0 & 0 & 0 & 0
  \end{array}
\right)\,.
$$
Then one considers the extended stacky fan whose corresponding matrix presentation for $\beta^{S}:\mathbb{Z}^{4\oplus 4}\rightarrow N$ is
$$
\left(
  \begin{array}{cccc|cccc}
-1 & 1 & 0 & 0 & 0 & 0 &0 & 0\\
-1 & 0& 1 & 0 & 0 & 0 &0 & 0\\
-1 & 0 & 0 & 1 & 0 & 0 &0 & 0\\
\hline
0 & 0 & 0 & 0 & 0 & 1 &2& 3
  \end{array}
\right)\,.
$$

With the twisting data, it follows that the twisted $I$-function is given by
\begin{equation}
I_{\mathrm{tw}}=\mathfrak{z} e^{ P\sum_{\rho=0}^{3}\mathfrak{t}_{\rho}/\mathfrak{z}}\sum_{b\in \{0,{1\over 4}, {2\over 4}, {3\over 4}\}}
\sum_{\substack{\lambda=(d,k_0,k_1,k_2,k_3)\geq 0\\
\langle {k_1+2k_2+3k_3\over 4}\rangle=b}}
\Lambda^{d}
e^{d\sum_{\rho=0}^{3}\mathfrak{t}_{\rho}}
x_0^{k_0}x_{1}^{k_{1}}x_{2}^{k_{2}}x_{3}^{k_{3}}I_{\lambda,b}M_{\lambda,b}\mathbb{1}_{b}\,,
\end{equation}
with
\begin{equation}
I_{\lambda,b}={1\over \mathfrak{z}^{k_{0}+k_{1}+k_{2}+k_{3}} k_{0}!k_{1}!k_{2}!k_{3}!}
\prod_{\ell=1}^{d}{1\over (P+\ell \mathfrak{z})^4}\,,\quad
M_{\lambda,b}=
\prod_{\ell=1}^{4d} (4P+\ell \mathfrak{z})
\,.
\end{equation}
Here
\begin{itemize}
\item
$P$ is the class of $c_{1}(\mathcal{O}_{Y_{\mathrm{min}}}(1))$ in $H^{2}(Y_{\mathrm{min}})$, $\mathfrak{t}_{\rho}\in \mathbb{C}$ with $\mathrm{Re}\,\mathfrak{t}_{\rho}<0$, $\rho=0,1,2,3$.
\item
$\Lambda$ is again  the Novikov variable, similarly $x_{j},j=0,1,2,3$ are formal parameters, and
 $\mathfrak{z}$ is the descendant variable.
\item
$b$ labels different components of the inertia stack $IY_{\mathrm{min}}$ of $Y_{\mathrm{min}}$,
$\mathbb{1}_{b}$
is the fundamental class for the corresponding component.
\end{itemize}
It is direct to see that
the summation in $I_{\mathrm{tw}}$ splits into the product of two, involving $d$ and $(k_0,k_1,k_2,k_3)$ respectively,
\begin{equation}
I_{\mathrm{tw}}=I_{Q}\cdot \sum_{(k_{0},k_{1},k_{2},k_{3})\in \mathbb{N}^{4}}
\prod_{j=0}^{3} {x_{j}^{k_{j}}\over \mathfrak{z}^{k_{j}} k_{j}!}\mathbb{1}_{\langle {k_{1}+2k_{2}+3k_{3}\over 4}\rangle}
\,,
\end{equation}
where $I_{Q}$ is the $I$-function for the quartic K3 given in \eqref{eqnIquarticK3}.
This reflects the product structure
$Y_{\mathrm{min}}=\mathbb{P}^{3}\times B\mathbold{\mu}_{4}$.
From the quantum Lefschetz principle
one can then determine the big $I$-function
$I_{Y_{\mathrm{min}}}$
 and
the corresponding $J$-function $J_{Y_{\mathrm{min}}}$.

 \subsubsection{Maximal quotient}

Consider the action of $G_{\mathrm{max}}=(\mathbb{Z}/4\mathbb{Z})^{4}$ on the homogeneous coordinate ring $\mathbb{C}[x,y,z,w]$ of $\mathbb{P}^{3}$
\begin{equation}
(k_0,k_1,k_2,k_3)\in G_{\mathrm{max}} : \,[x,y,z,w]\mapsto [e^{2\pi i k_0 \over 4}x, e^{2\pi i k_1 \over 4}y, e^{2\pi i k_2 \over 4}z,e^{2\pi i k_3 \over 4}w ]\,.
\end{equation}
We consider the faithful action by $G=G_{\mathrm{max}}/G_{\mathrm{min}}$.
The quotient
$Y_{\mathrm{max}}=Q/G$ is then a hypersurface in the toric stack $\mathbb{P}^{3}/G$.
Denote $i:Y_{\mathrm{max}}\rightarrow  \mathbb{P}^{3}/G$ to be the inclusion
and $IY_{\mathrm{max}}$ the inertia stack of $Y_{\mathrm{max}}$.

The number of components for the inertia stack $IY_{\mathrm{max}}$ is the same as
the cardinality of the Abelian group $G$.
These components can also be classified according to the
type of fixed locus $(\mathbb{P}^{3})^{(g)}$.
For example, let $z_0$ be the homogeneous coordinate corresponding to the ray $u_0$, then
all of the fixed loci for the reduced inertia stack can be chosen to be
intersections of coordinate hyperplanes with $z_0\neq 0$.
By examining the components $Y^{(g)}_{\mathrm{max}}$ in the inertia stack $IY_{\mathrm{max}}$
according to $g\in G$, one sees that
the coarse moduli are projective spaces and
 the dimension of $H_{\mathrm{CR}}^{*}=H^{*}(IY_{\mathrm{max}})$ is
\begin{equation}\label{eqnChenRuancohomologymaxquotientquartic}
\dim H^{*}(\mathbb{P}^{2})\cdot 1+\dim H^{*}(\mathbb{P}^{1})\cdot {4\choose 1}\cdot 3+\dim H^{*}(\mathbb{P}^{0})\cdot {4\choose 2}\cdot (6+ {3\over 2}\cdot 2)=81\,.
\end{equation}

For the stack $Y_{\mathrm{max}}$, the big quantum cohomology is defined on the Chen-Ruan cohomology
$H^{*}_{\mathrm{CR}}(Y_{\mathrm{max}})=H^{*}(IY_{\mathrm{max}})$.
According to \cite{Coates:2019}, the corresponding invariants are encoded in the $H^{*}_{\mathrm{CR}}([\mathbb{P}^{3}/G])$-valued $J$-function
which we now discuss.
First one takes the stacky fan $(N,\Sigma,\rho)$ that defines the toric Deligne-Mumford stack $[\mathbb{P}^{3}/G]$
\begin{equation*}
\beta:\mathbb{Z}^{4}\rightarrow N=\mathbb{Z}^{3}\,,
\end{equation*}
with
$\beta$ given by the matrix of rays in the fan
\begin{equation*}
(u_0,u_1,u_2,u_3)=
\left(
  \begin{array}{cccc}
-4 & 4 & 0 & 0\\
-4 & 0& 4 & 0\\
-4 & 0 & 0 & 4
  \end{array}
\right)\,.
\end{equation*}
Now following the method in \cite{Coates:2019}, one considers the extended stacky fan
and obtain the twisted $I$-function as before.
To be explicit, we take the extended fan to be given by $\beta^{S}:\mathbb{Z}^{4+4}\rightarrow N$
with the extra generators
mapped to $
{u_{j}/ 4}\in
\mathbb{Z}/4\mathbb{Z}\{u_0,u_1,u_2,u_3\}, j=0,1,2,3$.
We also take
\begin{equation*}
\mathrm{Box}(\Sigma)=
\{b=(b_{j})_{j=0}^{3}\in ({1\over 4}\mathbb{Z}/\mathbb{Z})^{\oplus 4} ~|~
\text{at least one}~b_{j} ~\text{is zero}\}\,.
\end{equation*}
The action of $G_{\mathrm{min}}$
induces the shift on $b=(b_{0},b_{1},b_{2},b_{3})$
by ${1\over 4}(1,1,1,1)$.
Note that elements in $\mathrm{Box}(\Sigma)$
are not in one-to-one correspondence with components of $IY_{\mathrm{max}}$, but modulo the action
by $G_{\mathrm{min}}$
this is so.
It follows that the twisted $I$-function is given by
\begin{equation}
\label{eqnK3maximalquotientI}
I_{\mathrm{tw}}=\mathfrak{z}e^{ P\sum_{\rho=0}^{3}\mathfrak{t}_{\rho}/\mathfrak{z}}\sum_{b\in \mathrm{Box}(\Sigma)}
\sum_{
\substack{d,k_{j}\geq 0\\
  k_{j}\in \mathbb{Z}\\
\langle  {1\over 4}k_{j} -d \rangle=b_j}
}
\Lambda^{d}
e^{d\sum_{\rho=0}^{3}\mathfrak{t}_{\rho}}
\prod_{j=0}^{3}x_{j}^{k_{j}}I_{\lambda,b}M_{\lambda,b}\mathbb{1}_{b}\,,
\end{equation}
with
\begin{eqnarray*}
I_{\lambda,b}&=&{1\over \prod_{j=0}^{3}\mathfrak{z}^{k_{j}}k_{j}!}
\prod_{j=0}^{3}
{\prod_{a:\langle a\rangle=\langle  d- {1\over 4} k_{j} \rangle,\, a
\leq 0 } (P+a\mathfrak{z})
\over
\prod_{a:\langle a\rangle=\langle  d-{1\over 4} k_{j} \rangle, \,a
\leq   d-  {1\over 4} k_{j} }(P+a\mathfrak{z})}
\,,\nonumber\\
M_{\lambda,b}&=&
\prod_{\ell=1}^{4d} (4P+\ell \mathfrak{z})
\,.
\end{eqnarray*}
Here again
\begin{itemize}
\item
$P$ is the class of $c_{1}(\mathcal{O}_{Y}(1))$ in $H^{2}(Y)$, $\mathfrak{t}_{\rho}\in \mathbb{C}$ with $\mathrm{Re}\,\mathfrak{t}_{\rho}<0$,
$\rho=0,1,2,3$.
\item
$\Lambda$ is the Novikov variable which we later set to $1$, $x_{j},j=0,1,2,3$ are formal parameters which will be collectively denoted by $x$, and
 $\mathfrak{z}$ is the descendant variable.
\item
$b$ labels different components  $Y^{(g)}_{\mathrm{max}},g\in G$ of the inertia stack $IY_{\mathrm{max}}$ of $Y_{\mathrm{max}}$,
$\mathbb{1}_{b}$
is the fundamental class for the corresponding component and has degree
$2\sum_{j=0}^{3} b_{j}$ in Chen-Ruan cohomology.
\end{itemize}

A direct calculation shows that
\begin{equation}
I_{\mathrm{tw}}=\mathfrak{z} F(\mathfrak{t})+\widetilde{G}(\mathfrak{t},x)
+\mathcal{O}(\mathfrak{z}^{-1})\,,
\end{equation}
where $F$ is the same as the one in \eqref{eqnFquarticK3} with the relation
$\mathfrak{t}=\sum_{\rho=0}^{3}\mathfrak{t}_{\rho}$,
and
\begin{equation*}
\widetilde{G}(\mathfrak{t},x)=G(\mathfrak{t})+F(\mathfrak{t})\sum_{j=0}^{3}x_{j}\mathbb{1}_{(0,\cdots, {1\over 4},\cdots, 0)}\,.
\end{equation*}
From the quantum Lefschetz principle
one has
$I_{Y_{\mathrm{max}}}(\mathfrak{t},x,\mathfrak{z})=i^{*} I_{\mathrm{tw}}$.
Similar to \eqref{eqnJquarticK3}, one has the $J$-function
\begin{equation*}
T(\mathfrak{t},x)={\widetilde{G}(\mathfrak{t}, x)\over F(\mathfrak{t})}\,,\quad
J_{Y_{\mathrm{max}}}(T,\mathfrak{z})={I_{Y_{\mathrm{max}}}(\mathfrak{t},x,\mathfrak{z})\over F(\mathfrak{t})}\,.
\end{equation*}

\subsection{Genus zero mirror symmetry}

We focus on the one-point functions of orbifold GW invariants
with fixed cohomology insertions.
They correspond to coefficients of the cohomology-valued $J$-function $J_{Y}=I_{Y}/F$, in
the Laurent expansion in $\mathfrak{z}^{-1}$ and with respect to
the basis for $H^{*}_{\mathrm{CR}}(Y)$.
For $Y=Q$ and $Y=Q/G_{\mathrm{min}}\,$, the identification between Givental $I$-functions and
automorphic forms is immediate
since the quantum differential equation
\eqref{eqnquantumdifferentialequationquarticK3} matches the Picard-Fuchs equation
\eqref{eqnPFK3beforetwist}
for the $D_{\bold{0}}$ sector.
We now discuss the maximal quotient $Y=Y_{\mathrm{max}}$.
Setting $\mathfrak{t}=\sum_{\rho=0}^{3} \mathfrak{t}_{\rho}, \varepsilon=P/\mathfrak{z}$ and simplifying
\eqref{eqnK3maximalquotientI}, one obtains
\begin{equation}\label{eqntwistedIfunctionexpandedinBrieskorn}
I_{\mathrm{tw}}
=\mathfrak{z}
\sum_{\bold{k}\in\mathbb{N}^4
}\prod_{j=0}^{3}{(\mathfrak{z}^{1\over 4} x_{j})^{k_{j}}\over \mathfrak{z}^{k_{j}}k_{j}!}
\cdot
I_{\mathrm{tw},\bold{k}}
\,,\quad
\bold{k}=(k_{0},k_{1},k_{2},k_{3})\,,
\end{equation}
with
\begin{eqnarray}
I_{\mathrm{tw},\bold{k}}(\mathfrak{t},\varepsilon,\mathfrak{z})&=&
e^{\varepsilon \mathfrak{t}}
\sum_{
\substack{d:\,
d\geq 0\\
 \prod_{j=0}^{3} \langle d-{1\over 4}k_{j} \rangle=0}
}
\Bigg(
e^{d\mathfrak{t}}
\prod_{j=0}^{3}
{\prod_{a:\langle a\rangle=\langle  d- {1\over 4} k_{j} \rangle, \,a
\leq 0 } (\varepsilon+a)
\over
\prod_{a:\langle a\rangle=\langle  d-{1\over 4} k_{j} \rangle,\, a
\leq   d-  {1\over 4} k_{j} }(\varepsilon+a)}
\prod_{\ell=1}^{4d} (4\varepsilon+\ell ) \label{eqnweightedtwistedsectorcomponentsofI}
 \\
&&\cdot\,\mathbb{1}_{\langle {1\over 4}\bold{k}-(d,d,d,d) \rangle}\prod_{j=0}^{3}\mathfrak{z}^{-\langle {1\over 4}k_{j}-d\rangle}
\Bigg)
\label{eqnfundamentalclass}\,.
\end{eqnarray}

Due to the constraint
 $\prod_{j=0}^{3} \langle d-{1\over 4}k_{j} \rangle=0$ in the range for
 $d$ in \eqref{eqnweightedtwistedsectorcomponentsofI},
 $I_{\mathrm{tw},\bold{k}}$ splits into the sum of $4$ series
according to the choice of $j_{*}\in \{0,1,2,3\}$ such that
$ \langle d-{1\over 4}k_{j_{*}}\rangle =0$.
Within each series
the quantity $\langle d-{1\over 4}k_{j}\rangle$ and thus
 \eqref{eqnfundamentalclass} are independent of
those $d$ that are subject to this condition.
In fact,
a closer look gives
\begin{eqnarray}
I_{\mathrm{tw},\bold{k}}(\mathfrak{t},\varepsilon,\mathfrak{z})&=&e^{\varepsilon \mathfrak{t}}
\sum_{
\substack{d:\,
d\geq 0\\
 \prod_{j=0}^{3} \langle d-{1\over 4}k_{j} \rangle=0}
}\Bigg( e^{dt}
 { \prod_{j=0}^{3} {\Gamma}(\varepsilon+1-\langle {1\over 4}k_{j}-d \rangle)\over \Gamma(4\varepsilon+1)}
{\Gamma(4\varepsilon+4d+1)\over  \prod_{j=0}^{3} \Gamma(\varepsilon+d-{1\over 4}k_{j}+1)  }
\label{eqnweightedtwistedsectorcomponentsofIsimplified}\\
&&\cdot \,\mathbb{1}_{\langle {1\over 4}\bold{k}-(d,d,d,d) \rangle}\prod_{j=0}^{3}\mathfrak{z}^{-\langle {1\over 4}k_{j}-d\rangle}\Bigg)\nonumber
\,.
\end{eqnarray}

From the perspective of the group action by $G$,
  each
$\bold{k}=(k_{0},k_{1},k_{2},k_{3})$
corresponds to
$$g:=(g_0,g_1,g_2,g_3)=(e^{2\pi i k_{0}\over 4},\cdots, e^{2\pi i k_{3}\over 4})\in \mathbold{\mu}_{4}^{\oplus 4}\cong G_{\mathrm{max}}\,.$$
The condition
 $\langle d-{1\over 4}k_{j_{*}}\rangle =0$ imposed by the constraint
 $\prod_{j=0}^{3} \langle d-{1\over 4}k_{j} \rangle=0$ corresponds to
 applying a transformation in $G_{\mathrm{min}}$
such that
the $g_{j_{*}}$-component of $g$ is gauged to $1$.
For such a choice of $j_{*}$, similar to the discussion in
\eqref{eqnChenRuancohomologymaxquotientquartic},
the coarse moduli of the corresponding inertia stack component $ Y^{(g)}$
is a projective space of dimension
$$3-\mathrm{cardinality}(\{j\in \{0,1,2,3\}~|~k_{j}-k_{j_{*}}=0\in \mathbb{Z}/4\mathbb{Z}\})\,.$$
Hence the cohomology ring $H^{*}(Y^{g})$ of this component  $Y^{(g)}$
gives a subspace of  the Chen-Ruan cohomology $H^{*}_{\mathrm{CR}}(Y)=H^{*}(IY)$ with dimension
\begin{equation}
N_{\bold{k}, j_{*}}:=4-\mathrm{cardinality}(\{j\in \{0,1,2,3\}~|~k_{j}-k_{j_{*}}=0\in \mathbb{Z}/4\mathbb{Z}\})\,.
\end{equation}
Specifically,  as a ring, its generators are given by
the fundamental class $\mathbb{1}_{\langle {1\over 4}\bold{k}-({1\over 4}k_{j_{*}},{1\over 4}k_{j_{*}},{1\over 4}k_{j_{*}},{1\over 4}k_{j_{*}}) \rangle}$ and the $P$-class.
Their contributions to the $H^{*}_{\mathrm{CR}}(Y)$-valued $I$-function
$i^{*}I_{\mathrm{tw}}$
are then obtained by
expanding the $d={1\over 4}k_{j_{*}}$ term in \eqref{eqnweightedtwistedsectorcomponentsofIsimplified}
with respect to the $\varepsilon$-parameter up to degree
$N_{\bold{k}, j_{*}}-1$.
The number of different functions arising as the coefficients in
\eqref{eqnweightedtwistedsectorcomponentsofIsimplified}
is thus given by
\begin{equation}\label{eqnreducedordertwistedsectorequationquartic}
N_{\bold{k}}:=4-\mathrm{cardinality}(\{k_{j} \}_{j=0}^{3})\,.
\end{equation}

From the expression in \eqref{eqnweightedtwistedsectorcomponentsofIsimplified}
one can directly check  that $I_{\mathrm{tw},\bold{k}}$
satisfies the differential equation
\begin{equation}\label{eqntwistedsectorcomponentsofI}
\left(\prod_{j=0}^{3}(\partial_{\mathfrak{t}}-{1\over 4} k_{j})
-4^{4}e^{\mathfrak{t}}\prod_{i=1}^{4} (\partial_{\mathfrak{t}}+{i\over 4})\right)I_{\mathrm{tw},\bold{k}}=0\,.
\end{equation}
Comparing \eqref{eqntwistedsectorcomponentsofI}
 with \eqref{eqnPicardFuchs} in Lemma \ref{lemPicardFuchs}, we see that
 $I_{\mathrm{tw},\bold{k}}$ satisfies the same equation as $a \cdot \zeta[\bold{z}^{\bold{m}}\Omega]$,
under the mirror map
\begin{equation}\label{eqnidentification}
 e^{\mathfrak{t}}=a^{-4}\,,\quad k_{j}=m_{j}\,,~j=0,1,2,3\,.
 \end{equation}
This tells that the $I$-function $I_{\mathrm{tw},\bold{k}}$
is (component of) a cohomology-valued automorphic form.
Under the map \eqref{eqnidentification}, the dimension
\eqref{eqnreducedordertwistedsectorequationquartic}
of the cohomology  ring of $Y^{(g)}$ agrees perfectly with the dimension
\eqref{eqnreducedordertwistedsectorequation}
of the space of twisted sectors involved in $D_{\bold{m}}$.
Interestingly,
the shift $\bold{k}\mapsto \bold{k}+(1,1,1,1)$
on the one hand is induced by the action of $G_{\mathrm{min}}$ on $Y$, on the other hand
maps the twisted sector $\zeta[\bold{z}^{\bold{m}}\Omega]$ to
 $\zeta[\bold{z}^{\bold{m}+\bold{1}}\Omega]$.
Moreover,  under \eqref{eqnidentification} we also see that
$I_{\mathrm{tw}}$ in \eqref{eqntwistedIfunctionexpandedinBrieskorn}
satisfies the same differential equation as the geometric section
\begin{equation}
a\cdot \mathfrak{z}
\sum_{\bold{m}\in\mathbb{N}^4
}\prod_{j=0}^{3}{(x_{j}\mathfrak{z}^{1\over 4})^{m_{j}}\over \mathfrak{z}^{m_{j}}m_{j}!}
\cdot
s[\bold{z}^{\bold{m}}\Omega]
=a \mathfrak{z}\, s [
e^{\sum_{j=0}^{3} { \mathfrak{z}^{1\over 4} x_{j}z_{j}\over\mathfrak{z}} }\Omega
]\,.
\end{equation}
According to the isomorphism in
 \eqref{eqncomparison},  this tells that $I_{\mathrm{tw}}$ is mirror to an oscillating integral, with
the fundamental class
\eqref{eqnfundamentalclass} recording the data $\bold{k}$ that is mirror to
$\bold{m}$. This matches the general expectations from mirror symmetry, see e.g., \cite{Givental:1995a, Givental:1995b, Givental:1996, Lee:2014, Iritani:2021}.

We therefore conclude that
the identification \eqref{eqnidentification} provides a map between twisted sectors in Chen-Ruan cohomology and
twisted sectors in singularity theory that respects the D-module structures.

\begin{appendices}

\section{Basics on singularities}
\label{secpreliminariesonsingularities}

We review some basics on deformations of singularities and fix some notation,
following closely \cite{Saito:1983, Matsuo:1998, Kulikov:1998, Saito:08}.
The results reviewed below admit coordinate free descriptions, however for concreteness we phrase them in local coordinates.

Let
\begin{equation}\label{eqngeneralsingularity}
f: (\mathbb{C}^{n+1},0)\rightarrow (\mathbb{C},0)
\end{equation}
be the germ of a holomorphic function near $0\in\mathbb{C}^{n+1}$ with only isolated critical point at $0\in \mathbb{C}^{n+1}$ 
and critical value 0. 
It is well-known that there is a sufficiently small disk \(\Delta\)
around \(0 \in \mathbb{C}\) and a sufficiently small neighborhood
\(B\) of \(0 \in \mathbb{C}^{n+1}\) such that \(f(B \cap f^{-1}(\Delta^{*})) \to \Delta^{*}\)
is a locally topologically trivial
fibration, called the Milnor fibration. Its typical fiber is called
the Milnor fiber.

Let $z=(z_{0},z_1,\cdots, z_{n})$ be the standard complex coordinates on $\mathbb{C}^{n+1}$ at the origin,
and
$t$ the one for the target space $\mathbb{C}$ near the critical value.
 If \(f\) has only isolated critical point at $0\in \mathbb{C}^{n+1}$,
then the Jacobi/Milnor ring
\begin{equation}
\mathrm{Jac}(f)=\mathbb{C}\{z_0,\cdots, z_n\}/\langle \partial_{z_{0}}f,\cdots, \partial_{z_{n}}f \rangle
\end{equation}
is a finite dimensional vector space whose dimension is given by the Milnor number $\mu=\mu(f)=\dim H^{n}(f^{-1}(t)),t\neq 0$.

Following \cite{Saito:1983, Matsuo:1998, Saito:08, Kulikov:1998}, we consider an $m$-parameter deformation, with $m\leq \mu$,
 with only isolated critical point at the origin as follows.
Taking a subset from the basis of the vector space $\mathrm{Jac}(f)$ which are represented by functions $e_{k}(z),k=0,1,\cdots, m-1$.
Consider the function
\begin{equation}
F(z,s)=f(z)+\sum_{k=0}^{m-1} s_{k}e_{k}(z)
\end{equation}
defined around a neighborhood $X$ of $(z,s)=(0,0)$ in $\mathbb{C}^{n+1}\times \mathbb{C}^{m}$, where $s=(s_0,\cdots, s_{m-1})$ are the standard coordinates
of $\mathbb{C}^{m}$ at the origin. Then the germ at $(z,s)=(0,0)$ of
$
F: X\rightarrow \mathbb{C}$
gives a deformation
of the germ of the function $f$.
Let
\begin{equation}
q: X\rightarrow S\subseteq \mathbb{C}^{m}
\,,\quad
\varphi=(F,q): X\rightarrow \mathbb{C}\times S\,,
\end{equation}
and
\begin{equation}
\pi_{1}: \mathbb{C}\times S\rightarrow  \mathbb{C}\,,\quad \pi_{2}: \mathbb{C}\times S\rightarrow S
\end{equation}
be the projections. Then one has the commutative diagram shown in
Figure \ref{fig:deformationofsingularity}.
\begin{figure}[h]
  \renewcommand{\arraystretch}{1}
\begin{displaymath}
\xymatrix{
&&X \ar[dd]^{\varphi=(F, q)} \ar[ddrr]^{q} \ar[ddll]_{F}
& & \\
&&       &&\\
\mathbb{C} && \mathbb{C}\times S \ar[ll]_{\pi_1} \ar[rr]^{\pi_2}
&& S            }
\end{displaymath}
 \caption[deformationofsingularity]{A deformation of singularity.}
   \label{fig:deformationofsingularity}
\end{figure}

Assume further that $S$ lies in the $\mu$-constant stratum.
Let
\begin{equation}
C=\{(z,s)\in X~|~\mathrm{rank} (d\varphi)_{(z,s)}\leq m\}
\end{equation}
be the critical set of the map $\varphi$. This is an analytic subspace of $X$ with
structure sheaf
\begin{equation}\label{eqncriticalset}
\mathcal{O}_{C}=\mathcal{O}_{X}/\langle \partial_{z_{0}}F,\cdots, \partial_{z_{n}}F\rangle\,.
\end{equation}
which has the usual commutative ring structure  $\circ$ and is regarded as a deformation of $\mathrm{Jac}(f)$.
By the assumption that $f$ has an isolated singularity, the map $q|_{C}: C\rightarrow S$ is a proper, finite map of degree $\mu$
and
thus $q_{*}\mathcal{O}_{C}$ is a locally free sheaf  of rank $\mu$ over $S$.

\subsection*{Frobenius structure}

The sheaf of relative differentials for the family $\varphi: X\rightarrow \mathbb{C}\times S$ is of central importance in studying
the mixed Hodge structure of singularities \cite{Kulikov:1998}.
Using the fact that $q$ is smooth, one can check that
$\Omega_{X/(\mathbb{C}\times S)}^{n+1}$ is a free $\mathcal{O}_{C}$-module of rank one and hence one has an identification
\begin{equation}\label{eqnprimitiveform}
\rho_{\xi }: q_{*}\mathcal{O}_{C}\rightarrow q_{*}\Omega_{X/(\mathbb{C}\times S)}^{n+1}\,.
\end{equation}
This identification is not canonical and requires a choice $\xi\in  q_{*}\Omega_{X/(\mathbb{C}\times S)}^{n+1}$
that is a generator of $q_{*}\Omega_{X/(\mathbb{C}\times S)}^{n+1}$ as a free $q_{*}\mathcal{O}_{C}$-module.
When $F$ is the semi-universal deformation (i.e., universal unfolding), with in particular $m=\mu$, a suitable choice of $\xi$ in \eqref{eqnprimitiveform} known as primitive form and some other data including higher residue parings
provide
a Frobenius algebra structure on $ q_{*}\mathcal{O}_{C}$, as well as a Frobenius manifold structure on $S$ via the Kodaira-Spencer map.
See  \cite{Saito:1983, Dubrovin:1996, Saito:08, Kato:1986, Satake:1993, Audin:1997, Satake:1998, Noumi:1998, Manin:1999fro, Manin:1999three,
Hertling:2002, Sabbah:2008, Satake:2010, Satake:2011, Strachan:2012}.
This structure can be further enhanced to a
special K\"ahler structure \cite{Strominger:1990, Freed:1999}.
Interested readers are referred to \cite{Dubrovin:1996, Cecotti:1991top, Alim:2017} for details.

\subsection*{Brieskorn lattice}
An essential ingredient in the studies of singularity theory
is the Brieskorn lattice $\mathcal{H}^{(0)}$ in the following set of sheaves
\begin{eqnarray}\label{eqnBrieskornlattices}
\mathcal{H}^{(-2)}&:=&\mathbb{R}^{n}\varphi_{*}\Omega_{X/(\mathbb{C}\times S)}^{\bullet}
\cong
\mathrm{Ker}\, (\varphi_{*}\Omega_{X/(\mathbb{C}\times S)}^{n}\rightarrow \varphi_{*}\Omega_{X/(\mathbb{C}\times S)}^{n+1})
/ d\varphi_{*}\Omega_{X/(\mathbb{C}\times S)}^{n-1}\,,\quad\nonumber\\
\mathcal{H}^{(-1)}&:=&
\varphi_{*}\Omega_{X/(\mathbb{C}\times S)}^{n}
/ d\varphi_{*}\Omega_{X/(\mathbb{C}\times S)}^{n-1}\cong
\varphi_{*}\Omega_{X/ S}^{n}
/(dF \wedge \varphi_{*}\Omega_{X  /S}^{n-1}+d\varphi_{*}\Omega_{X/S}^{n-1})
\,,\quad\nonumber\\
\mathcal{H}^{(0)}&:=&\varphi_{*}\Omega_{X/S}^{n+1}/(dF\wedge d\varphi_{*}\Omega_{X/ S}^{n-1})
\end{eqnarray}
where the isomorphisms follow from the Stein property of $\varphi$.
The latter two sheaves are extensions of the relative de Rham cohomology sheaf $\mathcal{H}^{(-2)}$.
Furthermore, one has
\begin{equation}\label{eqnfromBrieskornlatticetorelativedifferentials}
0\rightarrow \mathcal{H}^{(-1)}\rightarrow \mathcal{H}^{(0)}\xrightarrow{r} \varphi_{*}\Omega_{X/(\mathbb{C}\times S)}^{n+1}\rightarrow 0\,,
\end{equation}
where $r$ is the natural surjection.
The Brieskorn lattice $\mathcal{H}^{(0)}$ is
generated by the classes $[\omega]$ of the forms $\omega\in  \varphi_{*}\Omega_{X/S}^{n+1}$.
The corresponding
Gelfand-Leray residues, called geometric sections $s[\omega]$,
give rise to classes in $\mathcal{H}^{(-2)}$,
see \cite{Kulikov:1998}[Section I.5.1.6] and \cite{Arnold:1985, Arnold:1988}. The Brieskorn lattice $\mathcal{H}^{(0)}$ is
part of the Hodge filtration $\{\mathcal{H}^{(k)}\}_{k\geq 0}$ on the Gauss-Manin differential system $\mathcal{H}_{X}$ given by
\begin{equation}
\label{eqnHodgefiltrationGMsystem}
\mathcal{H}^{(k)}=F^{n-k}\mathcal{H}_{X}=\partial_{t}^{k}\mathcal{H}^{(0)}\,, \quad k\leq n\,.
\end{equation}
In particular, one has an inclusion
$\mathcal{H}^{(0)}\subseteq \mathcal{H}_{X}$.
The meromorphic connection $\mathcal{M}_{X}$ is related to $\mathcal{H}^{(0)}, \mathcal{H}_{X}$ by localization, that is, $\mathcal{M}_{X}=(\mathcal{H}_{X})_{(t)}=(\mathcal{H}^{(0)})_{(t)}$.
The D-modules $\mathcal{H}_{X},\mathcal{M}_{X}$ are equipped with the Kashiwara-Malgrange $V$-filtration $V^{\bullet}$ which can be defined
according the eigenvalues of the action by $t\partial_{t}$.
The canonical lattice $\mathcal{L}$ is then defined by
\begin{equation}
\mathcal{L}:=V^{>-1}\mathcal{M}_{X}\,.
\end{equation}
It is a locally free sheaf  that extends the  de Rham cohomology sheaf $\mathcal{H}^{(-2)}$
and lies in the Gauss-Manin differential system $\mathcal{H}_{X}$.

\subsection*{Mixed Hodge structure and embedding of the Brieskorn lattice}

We now consider mixed Hodge structures on the vanishing cohomology, following closely the expositions in
 \cite{Kulikov:1998}.

 Assume  that $S$ lies in the so-called $\mu$-constant stratum \cite{Arnold:1985, Arnold:1988}.
 Denote
  $$X(s)=X\times_{\mathbb{C}\times S} (\mathbb{C}\times \{s\} )\,,$$
then the restrictions of constructions for $X\rightarrow \mathbb{C}\times S$
to $X(s)\rightarrow \mathbb{C}$ coincide with
those for the singularity $f_{s}:=F|_{X(s)}$. To avoiding choosing $t\neq 0\in \mathbb{C}$ when considering structures on the cohomologies of
$X(s)_{t}:=f_{s}^{-1}(t), s\in S$, we take the homotopy fiber $X(s)_{\infty}$ and consider the vanishing cohomology $H^{n}(X(s)_{\infty})$
which can be intrinsically defined in terms of the vanishing cycle functor $R\Phi_{f_{s}}$
 as
  $$ H^{n}(X(s)_{\infty})=R^{n}\Gamma\circ R\Phi_{f_{s}}( \mathbb{C}_{X(s)})\,.$$
It follows that
the vanishing cohomology
$H^{n}(X(s)_{\infty})$ is isomorphic to the fiber  at $(0, s)$ of the canonical lattice
\begin{equation}
\label{eqninclusionofLinGMdifferentialsystem}
\mathcal{L}(s)\subseteq \mathcal{H}_{X(s)}\,.
\end{equation}
 Namely, one has
\begin{equation}
\label{eqnvanishingcohomologyasfiber}
\psi(s):
H^{n}(X(s)_{\infty})\cong \mathcal{L}(s)/t\mathcal{L}(s)\,.
\end{equation}
By analyzing the asymptotic behavior of the sections that  generate the Brieskorn lattice $\mathcal{H}^{(0)}$, one has the embedding
\begin{equation}
\label{eqninclusionofBrieskorninL}
\mathcal{H}^{(0)}(s)\subseteq \mathcal{L}(s)
\end{equation}
which
respects localization $(-)_{(t)}$.\\

The vanishing cohomology carries a mixed Hodge structure which we now describe.
From the monodromy action, one has
\begin{equation}\label{eqndecompositionofMHS}
  H^{n}(X(s)_{\infty})=\bigoplus_{\lambda:\,|\lambda|=1}H^{n}(X(s)_{\infty})_{\lambda}\xrightarrow{\underset{\sim}{\psi(s)}}\bigoplus_{-1<\alpha\leq 0} C_{\alpha}(s)\,,
\end{equation}
where
$H^{n}(X(s)_{\infty})_{\lambda}$ is the  eigenspace of the monodromy $T=T_{ss}T_{u}=T_{ss}e^{2\pi i N}$
of the singularity $f_{s}:\mathbb{C}^{n+1}\rightarrow \mathbb{C}$
around its singular fiber $f_{s}^{-1}(0)$ with eigenvalue $\lambda$.
This eigenspace is the same as  the generalized eigenspace  $C_{\alpha}(s)$ of $t\partial_t$
with eigenvalue $\alpha \in (-1,0]$.

The weight filtration is then induced by the Jacobson-Morosov filtration using the unipotent part $T_{u}$ of the monodromy $T$.
In particular, $H^{n}(X(s)_{\infty})_{\lambda=1}$ has weight $w=n+1$ and $H^{n}(X(s)_{\infty})_{\lambda\neq 1}$ weight $w=n$.

The Hodge filtration (in the sense of Saito-Scherk-Steenbrink) is induced by the one \eqref{eqnHodgefiltrationGMsystem} on the Gauss-Manin differential system $\mathcal{H}_{X}(s)$
via the inclusion \eqref{eqninclusionofLinGMdifferentialsystem} and the isomorphism \eqref{eqnvanishingcohomologyasfiber}.
It
is defined individually on each eigenspace $H^{n}(X(s)_{\infty})_{\lambda}$  or equivalently generalized eigenspace $C_{\alpha}(s)$ of $T$
as follows.
One first defines on the generalized eigenspace $C_{\alpha}(s)$
the filtration (recall that spaces in $\mathcal{H}^{(0)}(s)\subseteq \mathcal{L}(s)\subseteq \mathcal{H}_{X(s)}$ have induced $V$-filtrations from the one on $\mathcal{H}_{X(s)}$)
\begin{equation}\label{eqnFpC}
F^{p}C_{\alpha}(s)=\partial_{t}^{n-p}\mathrm{gr}_{V}^{n-p+\alpha}\mathcal{H}^{(0)}(s)\,,
\end{equation}
with $\mathrm{gr}_{V}^{n-p+\alpha}\mathcal{H}^{(0)}(s)$ generated by the leading parts of the sections
$[\omega],\omega\in {f_{s}}_{*}\Omega_{X(s)/\mathbb{C}}^{n+1}$.
This filtration satisfies (cf. \eqref{eqnfromBrieskornlatticetorelativedifferentials})
\begin{equation}
\mathrm{gr}_{F}^{p}C_{\alpha}(s)=\partial_{t}^{n-p}\mathrm{gr}_{V}^{n-p+\alpha}(\mathcal{H}^{(0)}(s)/\mathcal{H}^{(-1)}(s))\,.
\end{equation}
Recall that the spectra (or spectral numbers) of the singularity \(f_s\) is the
  multi-set of rational numbers \(\beta_1,\ldots,\beta_{\mu}\), with
  \begin{equation}\label{spectrum}
    \boxed{\operatorname{Card}\,\{i : \beta_i = \alpha + p\} = \dim\operatorname{Gr}^{n-p}_{F} C_{\alpha}(s)}\,,
  \end{equation}
  for \(p \in \mathbb{Z}\) and \(\alpha \in (-1,0]\).
Then via \eqref{eqndecompositionofMHS} one can transfer the above filtrations to the vanishing cohomology.
One declares
\begin{equation}\label{eqndecompositionofHodgefiltration}
   F^{p}H^{n}(X(s)_{\infty})\xrightarrow{\underset{\sim}{\psi(s)}} F^{p}(\mathcal{L}(s)/t\mathcal{L}(s))=\bigoplus_{-1<\alpha\leq 0} F^{p}C_{\alpha}(s)\,.
 \end{equation}
To be more concrete, the period integral of a geometric section $s[\omega]$ arising from
\begin{equation*}
[\omega]\in V^{>(n-p-1)}\mathcal{H}_{X(s)}\cap \mathcal{H}^{(0)}(s)=V^{>(n-p-1)}\mathcal{H}^{(0)}(s)
\end{equation*}
over a homology class $\gamma$ in the homological Milnor fibration  has the asymptotic behavior as $t\rightarrow 0$
\begin{equation}\label{eqnleadingpart}
\int_{\gamma} s[\omega]=t^{n-p+\alpha}t^{N}\omega_{\gamma}(s)+\cdots\,,\quad\alpha\in (-1,0]\,,\quad \omega_{\gamma}(s)\in\mathbb{C}\,.
\end{equation}
Fix a frame $\mathcal{B}$ of locally constant homology classes for the homological Milnor fibration.
Denote $\check{\gamma}$ to be the linear dual of $\gamma\in \mathcal{B}$ and
\begin{equation}\label{eqnzetaclass}
\zeta[\omega]:=\sum_{\gamma\in\mathcal{B}}\omega_{\gamma}(s)\cdot \check{\gamma}\,.
\end{equation}
The leading part of $[\omega]$ in $\mathrm{gr}_{V}^{n-p+\alpha}\mathcal{H}^{(0)}(s) $ gives rise to
a class in $F^{p}C_{\alpha}(s)$ via the map  $\partial_{t}^{n-p}$ in \eqref{eqnFpC},
and thus the class
$\zeta[\omega]$ in $F^{p}H^{n}(X(s)_{\infty})_{\lambda}$
under the further
isomorphism $\psi(s)$ in \eqref{eqndecompositionofMHS}.
We call $\zeta[\omega]$ the leading part of the geometric section $s[\omega]$.

The information of Hodge filtration is thus conveniently encoded in the spectrum of values $\beta=n-p+\alpha$ for the singularity
 defined in \eqref{spectrum}.\\

Since the monodromy operator $T$ is defined over $\mathbb{R}$,
the real structure on $H^{n}(X(s)_{\infty})$ maps $H^{n}(X(s)_{\infty})_{\lambda}$ to $H^{n}(X(s)_{\infty})_{\bar{\lambda}}$ in \eqref{eqndecompositionofMHS}
and can hence mix the spaces $C_{\alpha}(s)$ with different values of $\alpha$.
The polarization, the $N$ operator and the Hodge filtration $F^{\bullet}$
are all compatible with the eigenspace decomposition.

The embedding \eqref{eqninclusionofBrieskorninL} of the Brieskorn lattice
therefore contains crucial information of the mixed Hodge structure on the vanishing cohomology.
The period mapping that we use in this work is defined using this embedding.

\section{Elliptic modular forms in Calabi-Yau sector of Fermat quartic}
\label{secmarginalquartic}

We study elliptic modular forms
that appear in the  $D_{\bold{0}}$ sector of the Fermat quartic.

Under the Poincar\'e residue map \eqref{eqnLGCY}, the local system
 $D_{\bold{0}}$
corresponds to the Hodge bundle
$H_{T}^{2}(Q,\mathbb{C})$ of the Dwork pencil, where $H_{T}^{2}(Q,\mathbb{C})$ stands for the transcendental part
that is the orthogonal to the Neron-Severi lattice under the intersection pairing on $H^{2}(Q,\mathbb{Z})$.
The differential operator
$\mathcal{D}_{\mathrm{quartic}}$ in \eqref{eqnPFK3beforetwist}
then corresponds to the Picard-Fuchs operator for the Dwork family which is a lattice polarized
family of K3 surfaces.
It turns out that this polarization lattice is
$E_{8}^{\oplus 2} \oplus U\oplus \langle  -4 \rangle$, where $U$ stands for the rank-$2$ hyperbolic lattice with signature $(1,1)$.
The
transcendental lattice is then $H_{T}^{2}(Q,\mathbb{Z})\cong
U\oplus \langle 4\rangle$. See
\cite{Dolgachev:1996, Dolgachev:2013} for details.

\subsection*{Symmetric square structure and relation to elliptic modular forms}

From \cite{Dolgachev:1996} (see also \cite{Dolgachev:2013})
one sees that
\begin{equation}\label{eqnsymmetricquareHodgestructure}
H^{2}_{T}(Q,\mathbb{C})\cong \mathrm{Sym}^{\otimes 2 }H^{1}(E,\mathbb{C})\,,
\end{equation}
where $E$ is a certain universal elliptic curve family with level structure.
To be more precise,
the family $E$ is the one over the modular curve
$\Gamma_0(2)\backslash \mathcal{H}$
whose Picard-Fuchs operator
is given by
\begin{equation}\label{eqnPFellipticcurve}
\mathcal{D}_{\mathrm{elliptic}}:=\theta_{t}(\theta_{t}-{1\over 2})- t(\theta_{t}+{1\over 4})^{2}\,,\quad
\theta_{t}=t{\partial\over \partial t}\,,
\end{equation}
where $t$ is certain Hauptmodul on $\Gamma_0(2)\backslash \mathcal{H}$.
It is chosen such that the Fricke involution $W_{2}:\tau\mapsto -1/(2\tau)$ on $\mathcal{H}$
induces
\begin{equation}
W_{2}: t\mapsto {t\over t-1}\,.
\end{equation}
Solutions to \eqref{eqnPFellipticcurve} are vector-valued  elliptic modular forms for $\Gamma_{0}(2)$.
See \cite{Maier:2009, Zhou:2016mirror} for details.

The symmetric square structure in \eqref{eqnsymmetricquareHodgestructure}
can also be seen concretely \cite{Lian:1994, Lian:1996, Lian:1994zv} as follows.
\begin{itemize}
\item
First, by the Clausen identities \cite{Erdelyi:1981}[Section 4.3], the
differential operator \eqref{eqnPFK3beforetwist}
\begin{equation}\label{eqnPFK3beforetwistappendix}
\mathcal{D}_{\mathrm{quartic}}:=\theta_{b}(\theta_{b}-{1\over 4})(\theta_{b}-{2\over 4})
-b
(\theta_{b}+{1\over 4})^{3}\,,
\end{equation}
is the symmetric square of
\begin{equation}\label{eqnPFtriangle}
\mathcal{D}_{\mathrm{triangular}}:=\theta_{b}(\theta_{b}-{1\over 4})
-b
(\theta_{b}+{1\over 8})^{2}
\,.
\end{equation}

Consider instead of \eqref{eqnPFtriangle} the following operator, which
is  related to \eqref{eqnPFtriangle} by a twist by $b^{{1\over 8}}$ and has unipotent monodromy at $b=\infty$,
\begin{equation}\label{eqnshiftedtriangular}
b^{1\over 8}\circ \mathcal{D}_{\mathrm{triangular}}\circ b^{-{1\over 8}}
:=(\theta_{b}-{1\over 8})(\theta_{b}-{3\over 8})
-b
\theta_{b}^{2}\,.
\end{equation}
For the differential operator \eqref{eqnshiftedtriangular},
uniformization using Schwarzian \cite{Erdelyi:1981} equips \cite{Maier:2009, Dolgachev:1996, Dolgachev:2013} $\mathcal{M}$ with an orbifold structure
$$ \Gamma_{0}(2)^{+} \backslash \mathcal{H}^{*}-\{0,1,\infty\}\cong \mathcal{M}\,,$$
where $\Gamma_{0}(2)^{+}$
is the extension of the modular group $\Gamma_{0}(2)$ by the Fricke involution $W_{2}$.
The modular orbifold
$T_{4,2,\infty}=\Gamma_{0}(2)^{+} \backslash \mathcal{H}$
has signature $4,2,\infty$ at $b=0,1,\infty$  respectively.
In particular, up to the twist by $b^{1\over 8}$, solutions to \eqref{eqnPFtriangle} and thus to \eqref{eqnPFK3beforetwistappendix} are
automorphic forms for the group $\Gamma_{0}(2)^{+}$.
\item
The relation between solutions to
\eqref{eqnPFK3beforetwistappendix} and elliptic modular forms for
$\Gamma_{0}(2)$ is then established by
connecting solutions to \eqref{eqnPFtriangle} and
those to \eqref{eqnPFellipticcurve}.
This is achieved by making use of following quadratic relation for Gauss hypergeometric series \cite{Erdelyi:1981}[page 112]
\begin{equation}\label{eqnquadraticrelation}
_{2}F_{1}(\alpha,\beta;2\beta; t)
=
(1-t)^{-{1\over 2}\alpha}\,
_{2}F_{1}({1\over 2}\alpha,\beta-{1\over 2}\alpha; \beta+{1\over 2}; {t^{2}\over 4t-4})\,.
\end{equation}
To be more precise, identify the $\Gamma_{0}(2)^{+}$-invariant $b$
with the $W_{2}$-invariant function on $\Gamma_{0}(2) \backslash \mathcal{H}$ as
\begin{equation}\label{eqnW2invariant}
b={t^{2}\over 4t-4}: ~ \Gamma_{0}(2)\backslash \mathcal{H}\rightarrow \Gamma_{0}(2)^{+}\backslash \mathcal{H}\,.
\end{equation}
Then the two solutions to \eqref{eqnPFtriangle} given by
\begin{equation*} _{2}F_{1}({1\over 8},{1\over 8};{3\over 4}; b)\,,\quad
b^{1\over 4} \,_{2}F_{1}({3\over 8},{3\over 8};{5\over 4}; b)\,,
\end{equation*}
correspond via \eqref{eqnquadraticrelation} exactly to the
$(1-t)^{1\over 8}$-multiple of those
to  \eqref{eqnPFellipticcurve} given by
\begin{equation*}
 _{2}F_{1}({1\over 4},{1\over 4};{1\over 2}; t)\,,\quad
t^{1\over 2} \,_{2}F_{1}({3\over 4},{3\over 4};{3\over 2}; t)\,.
\end{equation*}

\item
It follows that solutions  to  \eqref{eqnPFK3beforetwistappendix}
are automorphic forms for $\Gamma_{4,2,\infty}\cong \Gamma_{0}(2)^{+}$ up to a twist by
 $b^{{1\over 8}}$
and are vector-valued elliptic  modular forms for $\Gamma(2)$ up to a further twist by $(1-t)^{1\over 8}$.
\end{itemize}

Either the Hodge-theoretic origin \eqref{eqnsymmetricquareHodgestructure} or
the concrete calculation on hypergeometric functions  tells that sections of $F^{3}D_{\bold{0}}\cong F^{3}\mathrm{gr}_{W}^{4}H$
are elliptic modular forms of weight two.

\subsection*{Integral presentation of monodromy group  from geometry}

Similar to the Fermat cubic case discussed in Section \ref{secexamples}, using the relation
to the Picard-Fuchs equation \eqref{eqnPFellipticcurve}  for the corresponding elliptic curve family, one can obtain
an integral representation for the monodromy group of \eqref{eqnPFK3beforetwistappendix} up to twist and conjugacy.
The details are given as follows.
In parallel to \eqref{eqnshiftedtriangular},
consider the twist of \eqref{eqnPFellipticcurve} given by
\begin{equation}\label{eqnPFellipticcurvetwist}
t^{1\over 4}\circ \mathcal{D}_{\mathrm{elliptic}}\circ t^{-{1\over 4}}
=(\theta_{t}-{1\over 4})(\theta_{t}-{3\over 4})- t\theta_{t}^{2}\,.
\end{equation}
The latter has unipotent monodromy at $t=1,\infty$.
Following \cite{Erdelyi:1981}, bases of solutions to \eqref{eqnPFellipticcurvetwist} and \eqref{eqnshiftedtriangular} are given by
\begin{eqnarray}\label{eqnbasisoftwistedelliptic}
u_{1}=~_{2}F_{1}({1\over 4},{3\over 4}; 1; t^{-1})\,,\quad
u_{2}=~_{2}F_{1}({1\over 4},{3\over 4};1; 1-t^{-1})\,.
\end{eqnarray}
and
\begin{eqnarray}\label{eqnbasisoftwistedtriangular}
v_{1}=~_{2}F_{1}({1\over 8},{3\over 8};1; b^{-1})
\,,\quad
v_{2}=~_{2}F_{1}({1\over 8},{3\over 8}; {1\over 2}; 1-b^{-1})\,,
\end{eqnarray}
respectively.
We also introduce a third solution to  \eqref{eqnshiftedtriangular} given by
\begin{eqnarray}\label{eqnextrasolutiontwistedtriangular}
v_{6}=(1-b^{-1})^{1\over 2}~_{2}F_{1}({7\over 8},{5\over 8}; {3\over 2}; 1-b^{-1})\,.
\end{eqnarray}
Identities among analytic continuations of hypergeometric series \cite{Erdelyi:1981}[page 65]
give
\begin{equation}
v_{1}=c_{1}v_{2}+c_{2}v_{6}\,,\quad
c_{1}={\Gamma({1\over 2})\over \Gamma({5\over 8})\Gamma({7\over 8})}\,,\quad
c_{2}={\Gamma(-{1\over 2})\over \Gamma({1\over 8})\Gamma({3\over 8})}\,.
\end{equation}
For readers' sake, we list solutions to various differential equations in Table \ref{tablevariousbases}.
\begin{table}[h]
  \centering
\caption{Solutions to various differential equations.}
  \label{tablevariousbases}
  \renewcommand{\arraystretch}{1.5} 
 \begin{tabular}{c|ccc}
 \hline
& $t^{1\over 4}\circ \mathcal{D}_{\mathrm{elliptic}}\circ t^{-{1\over 4}}$ &
$b^{1\over 8}\circ \mathcal{D}_{\mathrm{triangular}}\circ b^{-{1\over 8}} $ &
$b^{1\over 4}\circ \mathcal{D}_{\mathrm{quartic}}\circ b^{-{1\over 4}}$
 \\
&  \eqref{eqnPFellipticcurvetwist} &  \eqref{eqnshiftedtriangular} & \eqref{eqnsymmetricsquareofshiftedtriangular}\\
 \hline
solutions   & $u_{1},u_{2}$  & $v_{1},v_{2},\quad v_{6}$ \\
& \eqref{eqnbasisoftwistedelliptic}& \eqref{eqnbasisoftwistedtriangular}, ~ \eqref{eqnextrasolutiontwistedtriangular}   \\
   &  & $e_{1},e_{2}$ & $e_{1}^2,\, e_1 e_2, \, -2e_{2}^2$ \\
& & \eqref{eqnnewbasistwistedtriangular} &  \eqref{eqnbasistwistedK3} \\
\hline
\end{tabular}
\end{table}

Computations with the aid of modular forms
 \cite{Maier:2009} tell that
 for the elliptic curve family with the Picard-Fuchs operator given by \eqref{eqnPFellipticcurvetwist} one has
 $$\tau={\kappa_{E} u_2\over u_1}\,,\quad \kappa_{E}={i\over \sqrt{2}}\,.$$
and the monodromy group
for  \eqref{eqnPFellipticcurvetwist}
with respect to the basis $(u_1,\kappa u_2)$
is given by
$$T_{0,E}=
\begin{pmatrix}
-1 & -1\\
2 &1
\end{pmatrix}\,,
\quad
T_{1,E}=
\begin{pmatrix}
1 & 0\\
-2 &1
\end{pmatrix}\,,
\quad
T_{\infty,E}=
\begin{pmatrix}
1 & 1\\
0 &1
\end{pmatrix}\,,\quad T_{0,E}T_{1,E}T_{\infty,E}=\mathbb{1}\,.
$$
Here the subscript ``$E$'' stands for elliptic curve family.

Consider now a new fundamental basis of solutions to $b^{1\over 8}\circ \mathcal{D}_{\mathrm{triangular}}\circ b^{-{1\over 8}}$ in  \eqref{eqnshiftedtriangular}
\begin{equation}\label{eqnnewbasistwistedtriangular}
e=(e_{1},e_{2}):=(v_{1}, -\kappa_{E} v_{1}+\kappa v_{2})
=(c_{1} v_{2}+c_{2}v_{6}, \kappa_{E}(c_{1} v_{2}-c_{2}v_{6}))
\,,
\end{equation}
with $ \kappa=-{1\over 2\pi i} {\Gamma({1\over 8}) \Gamma({3\over 8})\over \Gamma({1\over 2}) }=2c_{1}\kappa_{E}$.
The following Schwarzian has the asymptotic behavior \cite{Erdelyi:1981} near $b= \infty$
$$s(0,{1\over 4},{1\over 2}; b^{-1}):= {\kappa  v_{2}\over v_{1}}
={1\over 2\pi i}\log ({b^{-1}\over 4^{4}})+{i\over \sqrt{2}}(1+\cdots)
\,.$$
Hence   in the basis $e$ the monodromy for \eqref{eqnshiftedtriangular} around $b=\infty$  is
$$T_{\infty,\mathrm{triangular}}=
\begin{pmatrix}
1 & 1\\
0 & 1
\end{pmatrix}\,.$$
In fact, near $b=\infty$, by lifting to the $t=\infty$ branch of \eqref{eqnW2invariant}, one has
\begin{equation*}
(e_{1},e_{2})=(u_{1},\kappa_{E} u_{2})\,.
\end{equation*}
The other branch around $t=1$ is related to this one by the Fricke involution $W_2$ that acts by $t^{-1}\mapsto 1-t^{-1}$ and
exchanges $u_{1},u_{2}$.
Around $b=1$, the local monodromy in the basis $(v_{2}, v_{6})$ is clearly given by the diagonal
matrix $\mathrm{diag}(1,-1)$.
It follows that the monodromies in the basis $e$ in \eqref{eqnnewbasistwistedtriangular} are given by
\begin{equation*}
T_{1,\mathrm{triangular}}=
\begin{pmatrix}
0 & {i\over \sqrt{2}}\\
{ \sqrt{2}\over i} & 0
\end{pmatrix}\,,\quad T_{0,\mathrm{triangular}}=(T_{1,\mathrm{triangular}}T_{\infty,\mathrm{triangular}})^{-1}=
\begin{pmatrix}
- {\sqrt{2}\over i} & {i\over \sqrt{2}}\\
{ \sqrt{2}\over i} & 0
\end{pmatrix}\,.
\end{equation*}

According to \cite{Dolgachev:1996, Dolgachev:2013},
the transcendental lattice $H_{T}^{2}(Q,\mathbb{Z})\cong
U\oplus \langle 4\rangle$
has the Gram matrix given by
\begin{equation}G=
\begin{pmatrix}
0 & 0 & 1\\
0 & 4 & 0\\
1& 0 &0
\end{pmatrix}\,.
\end{equation}
The presentation of the monodromy action on the
transcendental lattice should be integral and preserve the Gram matrix $G$, which we now check.
The relation between the quartic K3 in consideration and the product of two isogeneous elliptic curves from
the family $E$ above leads to a basis for the corresponding differential operator
\begin{equation}\label{eqnsymmetricsquareofshiftedtriangular}
b^{1\over 4}\circ \mathcal{D}_{\mathrm{quartic}}\circ b^{-{1\over 4}}=\mathrm{Sym}^{\otimes 2}
\left(
b^{1\over 8}\circ \mathcal{D}_{\mathrm{triangular}}\circ b^{-{1\over 8}}\right)
=(\theta_{b}-{1\over 4})(\theta_{b}-{2\over 4})(\theta_{b}-{3\over 4})
-b
\theta_{b}^{3}\,
\end{equation}
given by \cite{Dolgachev:1996, Dolgachev:2013} (cf. \cite{Nagura:1995, Hartmann:2013})
\begin{equation}\label{eqnbasistwistedK3}
(e_{1}^2,\, e_1 e_2, \, -2e_{2}^2) =(v_{1}^{2}, \kappa v_{1}v_{2}-\kappa_{E}v_{1}^2, -2\kappa^2 v_{2}^{2}+4\kappa \kappa_{E} v_{1}v_{2}+v_{1}^2)\,.
\end{equation}
Note that in the $t=\infty$ branch that covers $b=\infty$, this basis is lifted to
$(e_{1}^2, e_1 e_2, \, -2e_{2}^2) =(u_1^2, \kappa_{E} u_1 u_{2}, u_{2}^2)$ and in particular the lift of the normalized period there satisfies
\begin{equation}\label{eqnSchwarzianquarticCY}
s_{\bold{0}}:={e_{1}e_{2}\over e_{1}^2}={\kappa_{E} u_{2}\over u_{1}^2}=\tau\,.
\end{equation}
Then in this basis
the monodromies for \eqref{eqnsymmetricsquareofshiftedtriangular} around $0,1,\infty$ are given by
\begin{equation}\label{eqnintegralmonodromymarginalquartic}
 T_{0,\mathrm{K3}}=
\begin{pmatrix}
-2 &-1 & 1\\
4 & 1 &0\\
1 & 0 & 0
\end{pmatrix}\,,\quad
\quad
T_{1,\mathrm{K3}}=
\begin{pmatrix}
0 & 0 & 1\\
0 & 1 & 0\\
1& 0 &0
\end{pmatrix}\,,\quad
T_{\infty,\mathrm{K3}}=
\begin{pmatrix}
1 & 1 & -2\\
0 & 1 & -4\\
0& 0 &1
\end{pmatrix}\,.
\end{equation}
Since $TGT^{t}=G$,
they indeed preserve the Gram matrix as desired. Moreover, they
 have orders $4,2,\infty$ respectively, as should be the case
since  \eqref{eqnsymmetricsquareofshiftedtriangular}  corresponds to a representation of $\Gamma_{4,2,\infty}\cong \Gamma_{0}(2)^{+}$.

Another way to obtain integral presentations for the monodromy group is to follow the
computations on the hypergeometric group \cite{Beukers:1989}. The matrix presentation of the invariant Hermitian form however
is not given by the Gram matrix $G$ for the transcendental lattice $H_{T}^{2}(Q,\mathbb{Z})\cong
U\oplus \langle 4\rangle$ shown above.

\end{appendices}

\newcommand{\etalchar}[1]{$^{#1}$}
\providecommand{\bysame}{\leavevmode\hbox to3em{\hrulefill}\thinspace}
\providecommand{\MR}{\relax\ifhmode\unskip\space\fi MR }
\providecommand{\MRhref}[2]{%
  \href{http://www.ams.org/mathscinet-getitem?mr=#1}{#2}
}
\providecommand{\href}[2]{#2}

\bigskip{}


\noindent{\small 
Center for Mathematics and Interdisciplinary Sciences,
Fudan University, Shanghai 200433, P. R. China,\\ 
Shanghai Institute for Mathematics and Interdisciplinary
Sciences (SIMIS), Shanghai 200433, P. R. China,}

\noindent{\small Yau Mathematical Sciences Center, Tsinghua University, Beijing 100084, P. R. China}

\noindent{\small Email: \tt dingxinzhang@fudan.edu.cn}

\medskip{}

\noindent{\small Yau Mathematical Sciences Center, Tsinghua University, Beijing 100084, P. R. China}

\noindent{\small Email: \tt jzhou2018@mail.tsinghua.edu.cn}


\end{document}